\long\def\commentout#1{}

\newif\ifprint
\printfalse
\documentclass[
twoside,10pt]{HandbookOfModuli}

\usepackage{amsmath,amsthm}
\usepackage{amssymb}
\usepackage{latexsym}
\usepackage[utf8]{inputenc}
\usepackage{textcomp}
\usepackage{color}
\usepackage[pdftex]{graphicx}
\usepackage{titlesec}
\usepackage{xspace}

\ifprint
	\input{giovanni} 
	\usepackage[final]{microtype}
\fi
\usepackage{bm}
\renewcommand{\mathbf}[1]{\bm{#1}} 

\ifprint
	\definecolor{linkred}{rgb}{0,0,0} 
	\definecolor{linkblue}{rgb}{0,0,0} 
\else
	\definecolor{linkred}{rgb}{0.7,0.2,0.2}
	\definecolor{linkblue}{rgb}{0,0.2,0.6}
\fi

\numberwithin{equation}{section} 

\usepackage[
	hypertexnames=false,
	hyperindex,
	pagebackref,
	pdftex,
	pdftitle={Handbook of Moduli},
	pdfdisplaydoctitle,
	pdfpagemode=UseNone,
	breaklinks=true,
	extension=pdf,
	bookmarks=false,
	plainpages=false,
	colorlinks,
	linkcolor=linkblue,
	citecolor=linkred,
	urlcolor=linkred,
	pdfmenubar=true,
	pdftoolbar=true,
	pdfpagelabels,
	pdfpagelayout=TwoPageRight,
	pdfview=Fit,
	pdfstartview=Fit
]{hyperref}

\catcode`@=11 \baselineskip 4.5mm
\def\ps@handbook{\def\@oddhead{\hfill \leftmark \hfill\thepage }
\def\@evenhead{\thepage \hfill \rightmark \hfill}
\def\@oddfoot{}
\def\@evenfoot{}}
\def\@evenhead{}
\def\@oddfoot{}
\def\@evenfoot{\hfill\copyright\ China Higher Education Press}
\def\list#1#2{\ifnum \@listdepth >5\relax \@toodeep \else \global
\advance \@listdepth\@ne \fi \rightmargin \z@ \listparindent\z@
\itemindent\z@ \csname @list\romannumeral\the\@listdepth\endcsname
\def\@itemlabel{#1}\let\makelabel\@mklab \@nmbrlistfalse #2\relax
\@trivlist \parskip -\parsep \parindent\listparindent \advance
\linewidth -\rightmargin \advance\linewidth -\leftmargin \advance
\@totalleftmargin \leftmargin \parshape \@ne \@totalleftmargin
\linewidth \ignorespaces}
\renewcommand*\l@section{\@tocline{1}{0pt}{0em}{1.75em}{}}
\renewcommand*\l@subsection{\@tocline{2}{0pt}{1.75em}{2em}{}} 
\renewcommand\contentsnamefont{\bfseries\large} 
\catcode`@=12
\renewcommand{\theequation}{\thesection.\arabic{equation}}

\def\thebibliography#1{\section*{References}
\list{[\arabic{enumi}]}{\settowidth \labelwidth{[#1]} \leftmargin
\labelwidth \advance \leftmargin \labelsep \usecounter{enumi}}
\def\newblock{\hskip .11em plus .33em minus .07em} \sloppy
\clubpenalty 4000 \widowpenalty 4000 \sfcode`\.=1000 \relax}

\titleformat{\section}{\normalfont\large\bfseries}{\thesection.}{0.5em}{}[\kern0.em]
\titleformat{\subsection}{\normalfont\bfseries}{\thesubsection.}{0.3em}{}[\kern0.em]
\titleformat{\subsubsection}[runin]{\normalfont\bfseries}{\thesubsubsection.}{0.5em}{}[\kern0.5em]

\def\fofsubsubsection#1{\refstepcounter{equation}\subsubsection*{\theequation.\kern0.25em #1}}
\def\foisubsubsection#1{\refstepcounter{equation}\subsubsection*{\kern\parindent\theequation.\kern0.25em #1}}

\usepackage{amsmath,amssymb,amsthm,amsfonts,amscd,amsopn}
\usepackage{hyperref}
\usepackage{xspace}
\usepackage[all]{xy}\xyoption{dvips}

\DeclareMathAlphabet{\smallchanc}{OT1}{pzc}%
                                 {m}{it}
\DeclareFontFamily{OT1}{pzc}{}
\DeclareFontShape{OT1}{pzc}{m}{it}%
             {<-> s * [1.100] pzcmi7t}{}
\DeclareMathAlphabet{\mathchanc}{OT1}{pzc}%
                                 {m}{it}

\newcommand{\mcH}{\mathchanc{H}}

\newcommand{\mcL}{\mathchanc{L}}

\newcommand{\mcR}{\mathchanc{R}}

\newcommand{\mcm}{\mathchanc{m}}

\newcommand{\mco}{\mathchanc{o}}

\DeclareFontFamily{OMS}{rsfs}{\skewchar\font'60}
\DeclareFontShape{OMS}{rsfs}{m}{n}{<-5>rsfs5 <5-7>rsfs7 <7->rsfs10 }{}
\DeclareSymbolFont{rsfs}{OMS}{rsfs}{m}{n}
\DeclareSymbolFontAlphabet{\scr}{rsfs}

\newcommand{\sF}{\scr{F}}

\newcommand{\sI}{\scr{I}}

\newcommand{\sK}{\scr{K}}
\newcommand{\sL}{\scr{L}}
\newcommand{\sM}{\scr{M}}

\newcommand{\sO}{\scr{O}}

\newcommand{\sfA}{{\sf A}}
\newcommand{\sfB}{{\sf B}}
\newcommand{\sfC}{{\sf C}}
\newcommand{\sfD}{{\sf D}}

\newcommand{\sfF}{{\sf F}}
\newcommand{\sfG}{{\sf G}}

\newcommand{\sfK}{{\sf K}}

\newcommand{\bA}{\mathbb{A}}

\newcommand{\bC}{\mathbb{C}}

\newcommand{\bH}{\mathbb{H}}

\newcommand{\bN}{\mathbb{N}}

\newcommand{\bP}{\mathbb{P}}
\newcommand{\bQ}{\mathbb{Q}}

\newcommand{\bZ}{\mathbb{Z}}

\newcommand{\frT}{\mathfrak{T}}

\newcommand{\ol}{\overline}

\newcommand{\into}{\hookrightarrow}

\newcommand{\wtilde}{\widetilde}
\newcommand{\wt}{\widetilde}
\newcommand{\what}{\widehat}

\newcommand{\leteq}{\colon\!\!\!=}
\newcommand{\col}{\colon}

\newcommand{\rpforward}[1]{\myR{#1}_*}

\DeclareMathOperator{\codim}{codim}

\DeclareMathOperator{\opdiv}{div}

\DeclareMathOperator{\exc}{Exc}

\newcommand{\sHom}[0]{{\mcH\mco\mcm}}    

\DeclareMathOperator{\id}{{id}}
\DeclareMathOperator{\im}{{im}}

\DeclareMathOperator{\Ob}{{Ob}}

\DeclareMathOperator{\Proj}{{Proj}}

\DeclareMathOperator{\Sing}{{Sing}}

\DeclareMathOperator{\Spec}{{Spec}}
\DeclareMathOperator{\supp}{{supp}}

\newcommand{\factor}[2]{\left. \raise 2pt\hbox{\ensuremath{#1}} \right/
        \hskip -2pt\raise -2pt\hbox{\ensuremath{#2}}}

\newcommand{\myR}{{\mcR\!}}

\newcommand{\myL}{{\mcL}}

\newcommand{\blank}{\underline{\hskip 10pt}}

\newcommand{\bdot}{\kdot}

\newcommand{\kdot}{{{\,\begin{picture}(1,1)(-1,-2)\circle*{2}\end{picture}\ }}}
\newcommand{\mydot}{\kdot}

\newcommand{\DuBois}[1]{{\underline \Omega {}^0_{#1}}}
\newcommand{\FullDuBois}[1]{{\underline \Omega {}^{\mydot}_{#1}}}
\newcommand{\Om}{\underline{\Omega}}

\def\coh#1.#2.#3.{H^{#1}(#2,#3)}
\newcommand{\union}\cup
\newcommand{\intersect}\cap
\newcommand{\Union}\bigcup
\newcommand{\Intersect}\bigcap
\def\myoplus#1.#2.{\underset #1 \to {\overset #2 \to \oplus}}

\newcommand{\resto}{\big\vert_}

\def\qis{\,{\simeq}_{\text{qis}}\,}

\begin{document}
\makeatletter
\newenvironment{refmr}{}{}
%
\newcommand\james{M\raise .5ex \hbox{\text{c}}Kernan}

\renewcommand\thesubsection{\thesection.\Alph{subsection}}
\renewcommand\subsection{
  \renewcommand{\sfdefault}{pag}
  \@startsection{subsection}%
  {2}{0pt}{-\baselineskip}{.2\baselineskip}{\raggedright
    \sffamily\itshape\small
  }}
\renewcommand\section{
  \renewcommand{\sfdefault}{phv}
  \@startsection{section} %
  {1}{0pt}{\baselineskip}{.2\baselineskip}{\centering
    \sffamily
    \scshape
}}
\newcounter{lastyear}\setcounter{lastyear}{\the\year}
\addtocounter{lastyear}{-1}
\newcommand\sideremark[1]{%
\normalmarginpar
\marginpar
[
\hskip .45in
\begin{minipage}{.75in}
\tiny #1
\end{minipage}
]
{
\hskip -.075in
\begin{minipage}{.75in}
\tiny #1
\end{minipage}
}}
\newcommand\rsideremark[1]{
\reversemarginpar
\marginpar
[
\hskip .45in
\begin{minipage}{.75in}
\tiny #1
\end{minipage}
]
{
\hskip -.075in
\begin{minipage}{.75in}
\tiny #1
\end{minipage}
}}
\newcommand\Index[1]{{#1}\index{#1}}
\newcommand\inddef[1]{\emph{#1}\index{#1}}
\newcommand\noin{\noindent}
\newcommand\hugeskip{\bigskip\bigskip\bigskip}
\newcommand\smc{\sc}
\newcommand\dsize{\displaystyle}
\newcommand\sh{\subheading}
\newcommand\nl{\newline}
\newcommand\get{\input /home/kovacs/tex/latex/} 
\newcommand\Get{\Input /home/kovacs/tex/latex/} 
\newcommand\toappear{\rm (to appear)}
\newcommand\mycite[1]{[#1]}
\newcommand\myref[1]{(\ref{#1})}
\newcommand\myli{\hfill\newline\smallskip\noindent{$\bullet$}\quad}
\newcommand\vol[1]{{\bf #1}\ } 
\newcommand\yr[1]{\rm (#1)\ } 
\newcommand\cf{cf.\ \cite}
\newcommand\mycf{cf.\ \mycite}
\newcommand\te{there exist}
\newcommand\st{such that}
\newcommand\myskip{3pt}
\newtheoremstyle{bozont}{3pt}{3pt}%
     {\itshape}
     {}
     {\bfseries}
     {.}
     {.5em}
     {\thmname{#1}\thmnumber{ #2}\thmnote{ \rm #3}}
\newtheoremstyle{bozont-sf}{3pt}{3pt}%
     {\itshape}
     {}
     {\sffamily}
     {.}
     {.5em}
     {\thmname{#1}\thmnumber{ #2}\thmnote{ \rm #3}}
\newtheoremstyle{bozont-sc}{3pt}{3pt}%
     {\itshape}
     {}
     {\scshape}
     {.}
     {.5em}
     {\thmname{#1}\thmnumber{ #2}\thmnote{ \rm #3}}
\newtheoremstyle{bozont-remark}{3pt}{3pt}%
     {}
     {}
     {\scshape}
     {.}
     {.5em}
     {\thmname{#1}\thmnumber{ #2}\thmnote{ \rm #3}}
\newtheoremstyle{bozont-def}{3pt}{3pt}%
     {}
     {}
     {\bfseries}
     {.}
     {.5em}
     {\thmname{#1}\thmnumber{ #2}\thmnote{ \rm #3}}
\newtheoremstyle{bozont-reverse}{3pt}{3pt}%
     {\itshape}
     {}
     {\bfseries}
     {.}
     {.5em}
     {\thmnumber{#2.}\thmname{ #1}\thmnote{ \rm #3}}
\newtheoremstyle{bozont-reverse-sc}{3pt}{3pt}%
     {\itshape}
     {}
     {\scshape}
     {.}
     {.5em}
     {\thmnumber{#2.}\thmname{ #1}\thmnote{ \rm #3}}
\newtheoremstyle{bozont-reverse-sf}{3pt}{3pt}%
     {\itshape}
     {}
     {\sffamily}
     {.}
     {.5em}
     {\thmnumber{#2.}\thmname{ #1}\thmnote{ \rm #3}}
\newtheoremstyle{bozont-remark-reverse}{3pt}{3pt}%
     {}
     {}
     {\sc}
     {.}
     {.5em}
     {\thmnumber{#2.}\thmname{ #1}\thmnote{ \rm #3}}
\newtheoremstyle{bozont-def-reverse}{3pt}{3pt}%
     {}
     {}
     {\bfseries}
     {.}
     {.5em}
     {\thmnumber{#2.}\thmname{ #1}\thmnote{ \rm #3}}
\newtheoremstyle{bozont-def-newnum-reverse}{3pt}{3pt}%
     {}
     {}
     {\bfseries}
     {}
     {.5em}
     {\thmnumber{#2.}\thmname{ #1}\thmnote{ \rm #3}}
\newtheoremstyle{bozont-def-newnum-reverse-plain}{3pt}{3pt}%
   {}
   {}
   {}
   {}
   {.5em}
   {\thmnumber{\!(#2)}\thmname{ #1}\thmnote{ \rm #3}}
\theoremstyle{bozont}    
\newtheorem{skthm}[equation]{Theorem}
\newtheorem{skcor}[equation]{Corollary} 
\newtheorem{sklem}[equation]{Lemma} 
\newtheorem{skconj}[equation]{Conjecture}
\theoremstyle{bozont-sc}
\newtheorem{proclaim-special}[equation]{\specialthmname}
\newenvironment{proclaimspecial}[1]
     {\def\specialthmname{#1}\begin{proclaim-special}}
     {\end{proclaim-special}}
\renewcommand{\labelenumi}{{\rm (\theskthm.\arabic{enumi})}}
\newcounter{skeq}
\newcounter{skeqno}
\setcounter{skeq}{1}
\setcounter{skeqno}{1}
\numberwithin{skeq}{equation}
\theoremstyle{bozont-remark}
\newtheorem{skrem}[equation]{Remark}
\newtheorem{sksubrem}[equation]{Remark}
\newtheorem{skobs}[equation]{Observation} 
\newtheorem{skexample}[equation]{Example} 
\newtheorem{skclaim}[equation]{Claim} 
\newtheorem{skFact}[equation]{Fact}
\newtheorem*{SubHeading*}{\SubHeadingName}%
\newtheorem{SubHeading}[equation]{\SubHeadingName}
\newtheorem{sSubHeading}[equation]{\sSubHeadingName}
\newenvironment{demo}[1] {\def\SubHeadingName{#1}\begin{SubHeading}}
  {\end{SubHeading}}%
\newenvironment{subdemo}[1]{\def\sSubHeadingName{#1}\begin{sSubHeading}}
  {\end{sSubHeading}} %
\newenvironment{demo-r}[1]{\def\SubHeadingName{#1}\begin{SubHeading-r}}
  {\end{SubHeading-r}}%
\newenvironment{subdemo-r}[1]{\def\sSubHeadingName{#1}\begin{sSubHeading-r}}
  {\end{sSubHeading-r}} %
\newenvironment{demo*}[1]{\def\SubHeadingName{#1}\begin{SubHeading*}}
  {\end{SubHeading*}}%
\newtheorem{skdefini}[equation]{Definition}
%
\theoremstyle{bozont-reverse}    
\newtheorem{SubHeadingr}[equation]{\SubHeadingName}
\newenvironment{demor}[1]{\def\SubHeadingName{#1}\begin{SubHeadingr}}{\end{SubHeadingr}}
\theoremstyle{bozont-def-newnum-reverse}    
\newtheorem{newnumr}[equation]{}
\theoremstyle{bozont-def-newnum-reverse-plain}    
\newtheorem{newnumrp}[equation]{}
\makeatother
  
\title{Singularities of stable varieties}

\author{S\'andor J Kov\'acs}

\date{\today}

\thanks{Supported in part by NSF Grant 
  DMS-0856185, and the Craig McKibben and Sarah Merner Endowed Professorship in
  Mathematics at the University of Washington.}

\address{University of Washington, Department of Mathematics, 354350, Seattle, WA
  98195-4350, USA} 

\email{skovacs@uw.edu\xspace}

\urladdr{http://www.math.washington.edu/$\sim$kovacs\xspace}

\maketitle
\newcommand{\szabores}{Szab\'o-resolution\xspace}


\section{Introduction}

The theory of moduli of curves has been extremely successful and part of this success
is due to the compactification of the moduli space of smooth projective curves by the
moduli space of stable curves. A similar construction is desirable in higher
dimensions but unfortunately the methods used for curves do not produce the same
results in higher dimensions. In fact, even the definition of what \emph{stable}
should mean is not entirely clear a priori. In order to construct modular
compactifications of moduli spaces of higher dimensional canonically polarized
varieties one must understand the possible degenerations that would produce this
desired compactification that itself is a moduli space of an enlarged class of
canonically polarized varieties. 

The main purpose of the present article is to discuss the relevant issues that arise
in higher dimensions and how these lead us to the definition of stable varieties and
stable families. Particular emphasis is placed on understanding the singularities of
stable varieties including some recent results.

The structure of the article is the following: In \S\ref{sec:stable-curves} and
\S\ref{sec:canonical-models} I review the relevant properties of stable curves and
their families, including the admissible singularities of the total spaces of stable
families. In \S\ref{sec:stable-singularities} show how generalizing the properties of
the total spaces of stable families leads to the right generalization of stable
singularities in higher dimensions. In \S\ref{ssec:canonical-sheaves} I review the
construction and main properties of canonical sheaves and divisors.
\S\ref{sec:sing-mmp} is devoted to the singularities of the minimal model program,
mainly from a moduli theoretic point of view. In \S\ref{sec:duality-vanishing} I
define some important basic notions and recall some fundamental theorems such as
Grothendieck duality and Kodaira vanishing.  \S\ref{sec:rtl} and
\S\ref{sec:db-singularities} are concerned with the definition and basic properties
of rational and Du~Bois singularities respectively.  In
\S\ref{sec:splitting-principle} I review the most important criteria for rational and
Du~Bois singularities organized around the principle that a natural morphism in the
derived category should admit a left inverse essentially only if it is a
quasi-isomorphism combined by a push-forward map admitting a section by a trace map.
In \S\ref{sec:stable-families} I review the applications of the results in
\S\ref{sec:splitting-principle} to stable families and in \S\ref{sec:deform-db-sing}
the state of knowlwedge about the deformation theory of stable singularities.

Without trying to be comprehensive, here is a list of relevant references on
background. In order to study higher dimensional varieties one should be familiar
with the main techniques of birational geometry. The standard reference for this is
\cite{KM98} and for some more recent results the reader may consult
\cite{Hacon-Kovacs10}. For moduli spaces of higher dimensional smooth varieties a
good reference is \cite{Viehweg95}. For moduli spaces of stable varieties one may
refer to \cite{Kollar85,KSB88,Kollar90}. A light introduction to the ideas involved
is contained in \cite{MR2483953}.

\begin{demo}  {Definitions and notation} %
  \label{def:main-defs}
  Let $k$ be an algebraically closed field of characteristic $0$.  Unless otherwise
  stated, all objects will be assumed to be defined over $k$. A \emph{scheme} will
  refer to a scheme of finite type over $k$ and unless stated otherwise, a
  \emph{point} refers to a closed point.
  
  For a morphism $f:Y\to S$ and another morphism $T\to S$, the symbol $Y_T$ will
  denote $Y\times_S T$. In particular, for $t\in S$ I will write $Y_t = f^{-1}(t)$.
  In addition, if $T=\Spec F$, then $Y_T$ will also be denoted by~$Y_F$.
  
  Let $X$ be a scheme and $\sF$ an $\sO_X$-module. The
  \emph{$m^\text{th}$ reflexive power} of $\sF$ is the double dual (or
  reflexive hull) of the $m^\text{th}$ tensor power of $\sF$:
  $$
  \sF^{[m]}\leteq (\sF^{\otimes m})^{**}.
  $$
  A \emph{line bundle} on $X$ is an invertible $\sO_X$-module. A \emph{$\bQ$-line
    bundle} $\sL$ on $X$ is a reflexive $\sO_X$-module of rank $1$ one of whose
  reflexive power is a line bundle, i.e., there exists an $m\in \bN_+$ such that
  $\sL^{[m]}$ is a line bundle.  The smallest such $m$ is called the \emph{index} of
  $\sL$.

  For the advanced reader: whenever I mention {Weil divisors}, assume that $X$ is
  $S_2$ and think of a \emph{Weil divisorial sheaf}, that is, a rank $1$ reflexive
  $\sO_X$-module which is locally free in codimension $1$. For flatness issues
  consult \cite[Theorem 2]{kollar-hulls-and-husks}.

  For the novice: whenever I mention Weil divisors, assume that $X$ is normal and
  adopt the definition \cite[p.130]{Hartshorne77}. For the adventurous novice: This
  is mainly interesting for canonical divisors. Read \S\ref{ssec:canonical-sheaves}.

  For a Weil divisor $D$ on $X$, its associated \emph{Weil divisorial sheaf} is the
  $\sO_X$-module $\sO_X(D)$ defined on the open set $U\subseteq X$ by the formula
  \begin{multline*}
    \Gamma (U, \sO_X(D))=\left\{ \frac ab \ \bigg|\ a,b\in \Gamma(U, \sO_X), b \text{
        is not a zero divisor} \right. \\ \left. \text{anywhere on $U$, and } 
      \phantom{\bigg|} D+\opdiv(a)-\opdiv(b)\geq 0 \right\}
  \end{multline*}
  and made into a sheaf by the natural restriction maps.

  A Weil divisor $D$ on $X$ is a \emph{Cartier divisor}, if its associated Weil
  divisorial sheaf, $\sO_X(D)$ is a line bundle. If the associated Weil divisorial
  sheaf, $\sO_X(D)$ is a $\bQ$-line bundle, then $D$ is a \emph{$\bQ$-Cartier
    divisor}. The latter is equivalent to the property that there exists an $m\in
  \bN_+$ such that $mD$ is a Cartier divisor.
  
  The symbol $\sim$ stands for \emph{linear} and $\equiv$ for
  \emph{numerical equivalence} of divisors.

  Let $\sL$ be a line bundle on a scheme $X$. It is said to be
  \emph{generated by global sections} if for every point $x\in X$
  there exists a global section $\sigma_x\in \coh 0.X.\sL.$ such that
  the germ $\sigma_x$ generates the stalk $\sL_x$ as an
  $\sO_X$-module. If $\sL$ is generated by global sections, then the
  global sections define a morphism 
  $$
  \phi_{\sL}\col X\to\bP^N= \bP\left(\coh 0.X.\sL.\right).
  $$
  $\sL$ is called \emph{semi-ample} if $\sL^{m}$ is generated by global sections for
  $m\gg 0$. $\sL$ is called \emph{ample} if it is semi-ample and $\phi_{\sL^m}$ is an
  embedding for $m\gg 0$. A line bundle $\sL$ on $X$ is called \emph{big} if the
  global sections of $\sL^{m}$ define a rational map $\phi_{\sL^m}\col X\dasharrow
  \bP^N$ such that $X$ is birational to $\phi_{\sL^m}(X)$ for $m\gg 0$. Note that in
  this case $\sL^m$ is not necessarily generated by global sections, so
  $\phi_{\sL^m}$ is not necessarily defined everywhere. I will leave it to the reader
  the make the obvious adaptation of these notions for the case of $\bQ$-line
  bundles.

  If it exists, then a \emph{canonical divisor} of a scheme $X$ is denoted by $K_X$
  and the \emph{canonical sheaf} of $X$ is denoted by $\omega_X$. See
  \S\ref{ssec:canonical-sheaves} for more.

  A smooth projective variety $X$ is of \emph{general type} if
  $\omega_X$ is big.  It is easy to see that this condition is
  invariant under birational equivalence between smooth projective
  varieties. An arbitrary projective variety is of \emph{general type}
  if so is a desingularization of it.

  A projective variety is \emph{canonically polarized} if $\omega_X$
  is ample. Notice that if a smooth projective variety is canonically
  polarized, then it is of general type.

  Further definitions will be given in later sections. In particular, for the
  definition of \emph{Cohen-Macaulay} and \emph{Gorenstein} see
  \S\ref{ssec:canonical-sheaves}.
\end{demo}

\section{Stable curves}\label{sec:stable-curves}

First I will recall the definition and main propeties of families of stable curves
and then subsequently investigate how these may be generalized to higher dimensions.

\begin{skdefini}\cite[2.12]{MR1631825}\label{def:stable-curves}
  A \emph{stable curve} is a connected projective curve that 
  \begin{enumerate}
  \item has only nodes as singularities; and,
  \item has only finitely many automorphisms.\label{item:1}
  \end{enumerate}
\end{skdefini}

\medskip\noindent The finiteness condition on the automorphism group is equivalent to
either one the following:

\begin{enumerate}
\item[(\ref{def:stable-curves}.\ref{item:1}a)] Every smooth rational component of
  the curve meets the other components in at least $3$ points.
\item[(\ref{def:stable-curves}.\ref{item:1}b)] The dualizing sheaf of the curve is
  ample.
\end{enumerate}

With respect to (\ref{def:stable-curves}.\ref{item:1}b) note that nodes are local
complete intersections and hence a stable curve is Gorenstein by definition. In
particular its dualizing sheaf exists and it is a line bundle, and hence it makes
sense to ask whether it is ample.

The fact that the moduli functor of stable curves gives a good compactification of
the moduli functor of smooth curves hinges on the stable reduction theorem:

\begin{skthm}
  \cite[3.47]{MR1631825},\cite{KKMS73} Let $B$ be a smooth curve, $0\in B$ a point,
  and $B^\circ=B\setminus\{0\}$. Let $X^\circ\to B^\circ$ be a flat family of stable
  curves of genus $\geq 2$. Then there exists a branched cover $B'\to B$ totally
  ramified over $0$ and a family $X'\to B'$ of stable curves extending the fiber
  product $X^\circ\times_{B^\circ}B'$. Moreover, any two such extensions are
  dominated by a third. In particular, their special fibers, that is, the preimage of
  $0$ in $B'$, are isomorphic.
\end{skthm}

Note that being a family of stable curves implies that $X'\to B'$ does not have any
multiple fibers. On the other hand, one cannot expect to have a smooth total space,
$X'$, for this family, although its singularities are the mildest possible: In
general $X'$ will have Du Val singularities (of type $A$). This follows from an
explicit computation of the versal deformation space of a node.  These singularities
may be resolved by successive blowing ups resulting in an exceptional divisor
consisting of a chain of rational curves, each appearing with multiplicity $1$ in the
fiber of the blown up surface over the point $0\in B$. This leads to semi-stable
reduction where one only requires the curves in the family to be semi-stable, that
is, instead of (\ref{def:stable-curves}.\ref{item:1}a) one only requires that every
smooth rational component of the curve meets the other components in at least $2$
points, but in exchange one obtains that one may require the total space of the
family be smooth.

In the next statement I collect the ideas from these observations that will be
important in our quest to understand stable varieties in higher dimensions.

\begin{skobs}\label{obs:otal-space}
  For a stable family of curves, $X\to B$, let $\wt X\to X$ be a resolution of
  singularities and $0\in B$ a point. Then
  \begin{enumerate}
  \item $\omega_{X/B}$ is relatively ample;\label{item:2}
  \item the special fiber $X_0$ is uniquely determined by the rest of the
    family;\label{item:3} 
  \item $X$ has Du Val singularities; \label{item:4} and
  \item $\wt X\to B$ has reduced fibers; \label{item:5}
  \end{enumerate}
\end{skobs}

\section{Canonical models}\label{sec:canonical-models}

Next we will investigate how stability may be generalized to higher dimensions. For a
more detailed study and many other results see \cite{KSB88}.

First let us consider our goals. One wants to find a class of singularities that
allows us to define a moduli functor that would compactify the moduli functor of
smooth canonically polarized varieties. In other words, one wants to define
\emph{stable} varieties as canonically polarized varieties with singularities only
from this particular class and one would like that any family of smooth canonically
polarized varieties over a punctured curve have a unique stable limit, possibly over
a branched covering which is totally ramified over the punctured point.

Taking into account previous observations in the case of families of curves, this
means that one would like to achieve a notion of stable families such that
(\ref{obs:otal-space}.\ref{item:2}) and (\ref{obs:otal-space}.\ref{item:3}) remain
true. The first of these conditions is simply saying that stable varieties should be
canonically polarized. This is both reasonable and expected and if one is familiar
with the construction of moduli spaces via the Hilbert scheme (see for instance
\cite{Viehweg95}) then one can see that this is also necessary for other reasons as
well. The second condition, that is, uniqueness of specialization is important with
regard to the moduli space one hopes to construct eventually: this condition is
essentially saying that this moduli space would be separated, surely a condition one
would like to have.

The other two conditions in \eqref{obs:otal-space}, namely
(\ref{obs:otal-space}.\ref{item:4}) and (\ref{obs:otal-space}.\ref{item:5}) are
actually the ones that will help us figure out the right class of singularities
having the desired properties mentioned above.

It turns out that (\ref{obs:otal-space}.\ref{item:2}) and
(\ref{obs:otal-space}.\ref{item:4}) combined implies the uniqueness of
specialization, that is, once one has (\ref{obs:otal-space}.\ref{item:2}), then
(\ref{obs:otal-space}.\ref{item:4}) actually implies
(\ref{obs:otal-space}.\ref{item:3}). The last condition,
(\ref{obs:otal-space}.\ref{item:5}) will be useful in determining what class of
singularities would the fibers need to have in order for the total space to have the
kind of singularities that are the appropriate generalization of Du Val singularities
in the case of families of curves. We will investigate this further in
\S\ref{sec:stable-singularities}.

\emph{Du Val singularities}, also known as \emph{rational double points}, or
\emph{canonical Gorenstein surface singularities} may be defined a number of ways,
see \cite{MR543555} for fifteen of these. The original definition of them is actually
the one that generalizes well to higher dimensions.

In the following I will need to use the canonical sheaf on singular varieties. If $X$
is Cohen-Macaulay, then a dualizing sheaf exists and the canonical sheaf may be
defined as that. For the definition in more general settings please see
\S\ref{ssec:canonical-sheaves}.

\begin{skdefini}\label{def:can-sings}
  Let $X$ be a normal variety and assume that it admits a canonical sheaf $\omega_X$
  which is a line bundle. (This holds for example if $X$ is Gorenstein). 
  Then $X$ has \emph{canonical} singularities if for a resolution of singularities
  $\phi:\wt X\to X$ one has the folllowing:
  \begin{equation*}
    {\phi^*\omega_X\subseteq \omega_{\wt X}.}
  \end{equation*}
  If $\dim X=2$, these are also called \emph{Du Val} singularities.
\end{skdefini}

\begin{skrem}
  The assumption that $\omega_X$ is a line bundle is in fact not necessary to define
  canonical singularities, but it makes the definition simpler. We will later extend
  the definition to a larger class.
\end{skrem}

Notice that the (injective) morphism $\phi^*\omega_X\to \omega_{\wt X}$ does not
always exist. However, if a non-zero morphism like that exists, then it is
necessarily injective cf.\ \eqref{lem:injective}

Even though such a morphism does not always exist, it is easy to see that it does if
$X$ is smooth. Indeed, in that case there exists a natural morphism induced by the
pull-back of differential forms $\phi^*\Omega_X\to \Omega_{\wt X}$ and taking
determinants implies the existence of a non-zero morphism $\phi^*\omega_X\to
\omega_{\wt X}$.

\begin{sklem}\label{lem:injective}
  Let $Y$ be an irreducible variety, $\sL$ and $\sF$ torsion-free sheaves on $Y$, and
  $\alpha:\sL\to \sF$ a non-zero morphism. If $\sL$ has rank $1$, then $\alpha$ must
  be injective.
\end{sklem}

\begin{proof}
  Let $\sK=\ker\alpha$ and $\sI=\im\alpha$. If $\alpha$ is non-zero, then $\sI$ is a
  non-zero subsheaf of the torsion-free $\sF$. Since $\sL$ is rank $1$ (at the
  general point of $Y$) it follows that so is $\sI$. Therefore $\alpha$ is
  generically injective which implies that $\sK$ is a torsion sheaf. However, $\sL$
  is also torsion-free and hence $\sK=0$.
\end{proof}

\begin{skcor}\label{cor:can-sings-def}
  Let $X$ be a normal variety and assume that it admits a canonical sheaf $\omega_X$
  which is a line bundle. If for a resolution of singularities $\phi:\wt X\to X$
  there exists a non-zero morphism $\phi^*\omega_X\to \omega_{\wt X}$, then $X$ has
  canonical singularities. In particular, if $X$ is smooth, then it has canonical
  singularities and in the definition of canonical singularities if the required
  condition holds for a single resolution of singularities, then it holds for all of
  them.
\end{skcor}

\begin{proof}
  Left to the reader.
\end{proof}

This leads us to another interesting condition that characterizes canonical
singularities of Gorenstein varieties.

\begin{sklem}\label{lem:can-sings}
  Let $X$ be a normal variety and assume that it admits a canonical sheaf $\omega_X$
  which is a line bundle and $\phi:\wt X\to X$ a resolution of singularities. Then
  the following are equivalent:
  \begin{enumerate}
  \item $X$ has canonical singularities;
  \item $\phi_*\omega_{\wt X}\simeq \omega_{X}$; and
  \item $\phi_*\omega^{\otimes m}_{\wt X}\simeq \omega^{\otimes m}_{X}$ for all
    $m\geq 0$.\label{lem:can-sings-3}
  \end{enumerate}
\end{sklem}

\begin{proof}
  First assume that $\phi_*\omega_{\wt X}\simeq \omega_{X}$. Notice that there always
  exists a natural morphism $\phi_*\omega_{\wt X}\to \omega_{X}$, which is injective
  by \eqref{lem:injective}, so this condition could be phrased by saying that ``the
  natural morphism $\phi_*\omega_{\wt X}\to \omega_{X}$ is surjective''. In fact, the
  point of the condition is that this isomorphism implies that there exists a
  non-zero morphism $\omega_{X}\to \phi_*\omega_{\wt X}$ and via adjointness of
  $\phi^*$ and $\phi_*$ that implies the existence of a non-zero morphism
  $\phi^*\omega_X\to \omega_{\wt X}$, which in turn implies that $X$ has canonical
  singularities.

  Now assume that $X$ has canonical singularities, that is, there exists an injective
  morphism $\phi^*\omega_X\to \omega_{\wt X}$. It follows that the line bundle
  $\omega_{\wt X}\otimes \phi^*\omega_X^{-1}$ corresponds to an effective Cartier
  divisor $E$ on $\wt X$, so one obtains the expression:
  $$ \omega_{\wt X}\simeq \phi^*\omega_X\otimes \sO_{\wt X}(E).$$
  Since $X$ is normal it also follows that $E$ is $\phi$-exceptional and hence 
  $$ \omega_{\wt X}\resto E\simeq \sO_{E}(E).$$
  Therefore for any $m\geq 0$ one has the following short exact sequence:
  \begin{equation*}
    \xymatrix{%
      0\ar[r] & \phi^*\omega_X^{\otimes m} \ar[r] & \omega^{\otimes m}_{\wt X}
      \ar[r] & \sO_{mE}(mE) \ar[r] & 0 }.  
  \end{equation*}
  In order to finish the proof one needs to prove that $\phi_*\sO_{mE}(mE)=0$. This
  is easy to prove for surfaces, since the fact that $E$ is exceptional implies that
  its self-intersection is negative, hence the sheaf $\sO_{mE}(mE)$ has no global
  sections.  The statement in arbitrary dimension follows by a simple induction on
  the dimension considering general hyperplane sections. For details see
  \cite[1-3-2]{KMM87}.
\end{proof}

Combining canonical singularities with canonical polarization leads to the notion of
\emph{canonical models}: 

\begin{skthm}
  Let $X$ be a variety with canonical singularities and $\phi:\wt X\to X$ a
  resolution of singularities.  Assume that $\omega_X$ is ample. Then $X$ is
  isomorphic to the \emph{canonical model} of $\wt X$. In particular one has that
  $$
  X\simeq \Proj \bigoplus_{m\geq 0}H^0(\wt X, \omega_{\wt X}^{\otimes m}).
  $$
\end{skthm}

\begin{proof}
  Since $\omega_X$ is ample, it follows easily that 
  $$
  X\simeq \Proj \bigoplus_{m\geq 0}H^0(X, \omega_{X}^{\otimes m}),
  $$
  and $H^0(X, \omega_{X}^{\otimes m})\simeq H^0(\wt X, \omega_{\wt X}^{\otimes m})$
  for any $m\geq 0$ by (\ref{lem:can-sings}.\ref{lem:can-sings-3}).
\end{proof}

The same proof provides a relative version of this statement:

\begin{skthm}\label{thm:can-models}
  Let $f:X\to B$ be a proper flat morphism and $\phi:\wt X\to X$ a resolution of
  singularities. Let $\wt f=f\circ\phi$ and assume that $X$ has canonical
  singularities, $B$ is a smooth curve and $\omega_{X/B}$ is relatively ample with
  respect to $f$.  Then one has a natural $B$-isomorphism
  $$ 
  X/B\simeq \big(\Proj_B \bigoplus_{m\geq 0} \wt f_*\omega^{\otimes m}_{\wt
    X/B}\big)/B.\  
  $$
\end{skthm}

\begin{proof}
  Since $\omega_{X/B}$ is relatively ample, it follows that
  $$
  X/B\simeq \big(\Proj_B \bigoplus_{m\geq 0} f_*\omega^{\otimes m}_{X/B}\big)/B,
  $$
  and $f_*\omega^{\otimes m}_{X/B}\simeq f_*\omega^{\otimes m}_{X}\otimes
  \omega_B^{-m} \simeq \wt f_*\omega^{\otimes m}_{\wt X}\otimes \omega_B^{-m} \simeq
  \wt f_*\omega^{\otimes m}_{\wt X/B}$ for any $m\geq 0$ by
  (\ref{lem:can-sings}.\ref{lem:can-sings-3}).
\end{proof}

\begin{skcor}
  Let $B$ be a smooth curve, $0\in B$ a point, and $B^\circ=B\setminus\{0\}$. Let $f:
  X\to B$ and $f': X'\to B$ be two proper flat morphisms such that restricting $f$
  and $f'$ over $B^\circ$ gives isomorphic families, i.e., $(X\times_B
  B^\circ)/B^\circ\simeq (X'\times_B B^\circ)/B^\circ$ as $B^\circ$-schemes. If both
  $X$ and $X'$ have canonical singularities and both $\omega_{X/B}$ and
  $\omega_{X'/B}$ are relatively ample, then $X/B\simeq X'/B$ as $B$-schemes. In
  particular, the special fibers of $f$ and $f'$ are isomorphic: $X_0\simeq X'_0$.
\end{skcor}

\begin{proof}
  Let $\wt X$ be a common resolution of singularities of $X$ and $X'$ with resolution
  morphisms be $\phi:\wt X\to X$ and $\phi':\wt X\to X'$. It follows that then
  $f\circ \phi =f'\circ\phi'$ so one may denote this morphism by $\wt f$ and so
  $$ 
  X/B\simeq \big(\Proj_B \bigoplus_{m\geq 0} \wt f_*\omega^{\otimes m}_{\wt
    X/B}\big)/B \simeq X'/B
  $$ by \eqref{thm:can-models}.
\end{proof}

The important conclusion to draw from this is that in order to guarantee uniqueness
of specialization one should require that a stable family has a relatively ample
canonical sheaf and its total space has canonical singularities.

\begin{skobs}\label{obs:total-space}
  For a stable family $X\to B$ over a smooth curve $B$ let $\wt X\to X$ be a
  resolution of singularities and $0\in B$ a point. Then one expects the following
  conditions to hold:
  \begin{enumerate}
  \item $\omega_{X/B}$ is relatively ample;\label{item:36}
  \item $X$ has canonical singularities; and \label{item:37}
  \item $\wt X\to B$ has reduced fibers;\label{item:38}
  \end{enumerate}
\end{skobs}

Notice that I dropped the condition that ``the special fiber $X_0$ is uniquely
determined by the rest of the family'' from \eqref{obs:otal-space} not because we no
longer need it but because (\ref{obs:total-space}.\ref{item:36}) and
(\ref{obs:total-space}.\ref{item:37}) imply it.

In the next section we will investigate what the third condition
(\ref{obs:total-space}.\ref{item:38}) gives us with regard to the singularities of
the fibers.

\section{Stable singularities}\label{sec:stable-singularities}

Let $f:X\to B$ be a flat morphism over a smooth curve $B$, $\phi:\wt X\to X$ a
resolution of singularities and $0\in B$ a point.  Assume that $X$ has canonical
singularities and $\wt f:\wt X\to B$ has reduced fibers.

One would like to understand the condition this places on the singularities of $X_0$,
the special fiber of $f$. To this end let us assume that $\phi$ is an embedded
resolution of $X_0\subset X$ and such that $\phi^{*}X_0=\wt X_0\subset \wt X$ is an
snc divisor.  Notice that by assumption $B$ is a smooth curve, so $X_0$ is a Cartier
divisor and hence pulling it back makes sense. Furthermore, assume that $\wt f$ has
reduced fibers so $\phi^{*}X_0=\wt X_0$ itself is an snc divisor not just that its
support is one.

  We saw in the proof of \eqref{lem:can-sings} that $X$ having canonical
  singularities implies, and in fact is equivalent to, that
  \begin{equation} 
    \label{eq:3}
    \omega_{\wt X}\simeq \phi^*\omega_X(E)
  \end{equation}
  for some effective $\phi$-exceptional divisor $E\subset \wt X$. 

  Since $\phi$ is an embedded resolution of $X_0\subset X$, $\wt X_0$ contains a
  union of components $\what X_0$ that gives a resolution of singularities
  $\what\phi_0=\phi\resto{\what X_0}:\what X_0\to X_0$.  One cannot, however, expect
  $\wt X_0$ be equal to $\what X_0$, so one obtains that
  \begin{equation} 
    \label{eq:4}
    \phi^*X_0=\wt X_0 = \what X_0 + F,
  \end{equation}
  where $F$ is the effective $\phi$-exceptional divisor formed by the unions of the
  components of $\wt X_0$ not contained in $\what X_0$. Since $\wt X$ is smooth, all
  of these are Cartier divisors.

  By adjunction one has that $\omega_{\what X_0}\simeq \omega_{\wt X}(\what
  X_0)\resto{\what X_0}$ and $\omega_{X_0}\simeq \omega_{X}(X_0)\resto{X_0}$.
  Combining this with (\ref{eq:3}) and (\ref{eq:4}) leads to the isomorphism
  \begin{multline*}
    \omega_{\what X_0}\simeq \omega_{\wt X}(\what X_0)\resto{\what X_0} \simeq
    \phi^*\omega_X(E+\phi^*X_0-F)\resto{\what X_0} \simeq \\
    \simeq {\what\phi_0}^*\left(\omega_X(X_0)\resto{X_0} \right)\otimes \sO_{\what
      X_0}\left((E-F)\resto{\what X_0}\right) \simeq {\what\phi_0}^*\omega_{X_0}\otimes
    \sO_{\what X_0}\left((E-F)\resto{\what X_0}\right)
  \end{multline*}
  Now let $\what E_0=E\resto{\what X_0}$ and $\what F_0=F\resto{\what X_0}$. Then one
  obtains that
  \begin{equation} 
    \label{eq:6}
    \what\phi_0^*\omega_{X_0} \subseteq \omega_{\what X_0}(\what F_0)
  \end{equation}
  This is not quite the definition of canonical singularities, but a somewhat weaker
  condition. Notice however that while we did not know much about the multiplicities
  of the components of $E$ other than that they are non-negative, we do know that
  $\wt X_0=\what X_0 +F$ is an snc divisor and hence so is $\what F_0=F\cap \what X_0
  \subset \what X_0$. This is an important detail. This means that although
  $\what\phi_0^*\omega_{X_0}$ does not necessarily admit a non-zero morphism to
  $\omega_{\what X_0}$, it does admit an embedding to a slightly larger sheaf.  This
  leads to the definition of \emph{(semi) log canonical singularities} see
  \S\ref{sec:sing-mmp} for more details.

  Observe that the above computation works backwards as well, so we actually found
  what we were looking for: a condition on the singularities of the fibers instead of
  a condition on the singularities of the total space.

\section{The dualizing sheaf versus the canonical divisor}\label{ssec:canonical-sheaves}

In order to construct moduli spaces one needs a polarization of our objects. The
(essentially only) natural choice of a line bundle on an abstract smooth projective
variety is the canonical bundle. This is the main reason we are studying
\emph{canonically} polarized varieties. When one extends our moduli problem in order
to have compact moduli spaces one still needs a canonical polarization. However, the
dualizing sheaf, even if it exists, is not necessarily a line bundle. Therefore, a
discussion of how one produces canonical polarizations on  stable varieties is in
order. Below we will use many of the notions and notation from \eqref{def:main-defs}
but we also need a few more.

\begin{skdefini}\label{def:CM}
  A finitely generated non-zero module $M$ over a noetherian local ring $R$ is called
  \emph{Cohen-Macaulay} if its depth over $R$ is equal to its dimension. For the
  definition of depth and dimension I refer the reader to \cite{MR1251956}. The ring
  $R$ is called \emph{Cohen-Macaulay} if it is a Cohen-Macaulay module over itself.

  Let $X$ be a scheme and $x\in X$ a point. One says that $X$ has
  \emph{Cohen-Macaulay} singularities at $x$ (or simply $X$ is \emph{CM} at $x$), if
  the local ring $\sO_{X,x}$ is Cohen-Macaulay.

  If in addition, $X$ admits a dualizing sheaf $\omega_X$ which is a line bundle in
  a neighbourhood of $x$, then $X$ is \emph{Gorenstein} at $x$.

  The scheme $X$ is \emph{Cohen-Macaulay} (resp.\ \emph{Gorenstein}) if it is
  \emph{Cohen-Macaulay} (resp.\ \emph{Gorenstein}) at $x$ for all $x\in X$.
\end{skdefini}

If $X$ is Cohen-Macaulay, then it admits a dualizing sheaf. However, stable varieties
are not necessarily Cohen-Macaulay, so one needs a more sophisticated approach.

Stable varieties are projective and projective varieties admit \emph{dualizing
  complexes}: If $X\subseteq \bP^N$ and $d=\dim X$, then
$$
\omega_X^\mydot \qis \myR\sHom_{\bP^N}(\sO_X,\omega_{\bP^N}[N]).
$$
Using this dualizing complex one can always define the \emph{canonical sheaf}:
$$
\omega_X\leteq h^{-d}(\omega_X^\mydot)
$$
In fact, this allows us to define the canonical sheaf of any quasi-projective
variety, or more generally any locally closed subset of a variety that admits a
dualizing complex. For more on this the reader is referred to
\cite{MR0222093,Conrad00}. 

Suppose $U\subseteq X$ is an open subset of the projective variety $X$. Then let
$$
\omega_U^\mydot\leteq {\omega_X^\mydot}\resto U.
$$

\begin{skrem}
  \label{rem:CM-and-Gor}
  Note that $X$ is Cohen-Macaulay if and only if $$\omega_X^\mydot\qis \omega_X[d],$$
  that is, if the only non-zero cohomology sheaf of $\omega_X^\mydot$ is the
  $-d^{\text{th}}$ (and $d$ still denotes $\dim X$). In this case the canonical sheaf
  is isomorphic to the \emph{dualizing sheaf}.

  $X$ is Gorenstein if and only if it is Cohen-Macaulay and $\omega_X$ is a line
  bundle. 
\end{skrem}

For a normal variety $X$ the usual way to define the canonical sheaf is different but
produces the same sheaf.  Being normal is equivalent to being $R_1$ and $S_2$, that
is, $X$ is normal if and only if it is non-singular in codimension $1$ and satisfies
Serre's $S_2$ condition.

Let $U=X\setminus\Sing X$ be the locus where $X$ is non-singular and $\iota:U\into X$
its natural embedding to $X$. Then one may define the canonical=dualizing sheaf of
$U$ as the determinant of the cotangent bundle, i.e., the sheaf of top differential
forms, $\omega_U=\det\Omega_U$. Then the usual definition of the canonical sheaf of
$X$ is $\omega'_X\leteq\iota_*\omega_U$.  It is relatively easy to see that both
$\omega_X$ and $\omega'_X$ are reflexive and agree in codimension $1$, so they are
actually isomorphic (cf.\ \cite[\S1]{MR597077}).
$$
\xymatrix{%
  \omega_X\ar[r]^{\simeq} &  \omega'_X \\
  h^{-d}(\omega_X^\mydot) \ar@{=}[u]
  & \iota_*\omega_U. \ar@{=}[u]
}
$$
Indeed, since $\iota:U\into X$ is an open embedding, the restriction of the dualizing
complex of $X$ to $U$ is the dualizing complex of $U$:
$$
{\omega_X^\mydot}\resto U\qis \omega_U^\mydot.
$$
In particular, since restriction to $U$ is an exact functor, ome also has
$$
{\omega_X}\resto U\simeq \omega_U.
$$
Recall that $X$ is assumed to be normal. In that case the $R_1$ condition implies
that $\codim_X(X\setminus U)\geq 2$ and the $S_2$ condition combined with the fact
that $\omega_X$ is reflexive implies that then
\begin{equation}
  \label{eq:7}
  \omega_X\simeq \iota_*\left({\omega_X}\resto U\right)\simeq \iota_*\omega_U.
\end{equation}
Possibly some readers are more familiar with this isomorphism in the divisor setting.

Let $X$ be an irreducible normal variety and $\iota:U\into X$ the non-simgular locus
as above.  A \emph{canonical divisor} $K_X$ of $X$ is a Weil divisor whose associated
\emph{Weil divisorial sheaf},
$$
\sO_X(K_X)\leteq \left\{ f\in K(X) \vert K_X + \opdiv(f) \geq 0 \right\},
$$ 
is isomorphic to the canonical sheaf $\omega_X$. This is usually defined the
following way: Define $\omega_U$ as above. As $U$ is non-singular, $\omega_U$ is a
line bundle and hence corresponds to a Cartier divisor. Let $K_U=\sum \lambda_iK_i$
denote a Weil divisor associated to this Cartier divisor.  Let $\ol K_i$ denote the
closure of $K_i$ in $X$ and let
$$
K_X\leteq \sum \lambda_i \ol K_i.
$$ 
Since $\codim_X(X\setminus U)\geq 2$, this is the unique Weil divisor on $X$ for
which ${K_X}\resto U= K_U$. By the same argument as in the paragraph preceding
\eqref{eq:7} it follows that 
$$
\omega_X\simeq \sO_X(K_X).
$$

As already clear from the case of curves, when working with objects on the boundary
of the moduli space one is forced to work with non-normal schemes. We will need one
more important detail to make this work. Notice that $U$ being non-singular is not
essential in the above constructions. Since one knows how to define the canonical
sheaf of a quasi-projective variety, one does not need $U$ to be non-singular for
that. The only place where we used the non-singularity of $U$ was to establish that
$\omega_U$ is a line bundle.  In other words, we may replace the condition of $U$
being non-singular with assuming that its canonical sheaf is a line bundle. In
particular, assuming that $U$ is Gorenstein will do the trick and then we still have
that
\begin{equation}
  \label{eq:canonical}
  \omega_X\simeq \iota_*\left({\omega_X}\resto U\right)\simeq \iota_*\omega_U.
\end{equation}

The precise condition we need in order to be able to define stabile varieties is the
following.

\begin{skdefini}\label{def:g1}
  A variety is called $G_1$ if it is Gorenstein in codimension $1$.
\end{skdefini}

If $X$ is $G_1$ and $S_2$ then everything said about the canonical sheaf of normal
varieties above works the same way. In particular, one may talk about a
\emph{canonical divisor} $K_X$ which is a Weil divisor that is Cartier in codimension
$1$. In fact, if $X$ is $G_1$ and $S_2$, then one does not need to assume that $X$
admits a dualizing complex and one does not need to define the canonical sheaf that way:

\begin{skdefini}
  \label{def:can-sheaf}
  Let $X$ be a scheme that is $G_1$ and $S_2$ and $\iota:U\into X$ be an open set
  such that $\codim_X(X\setminus U)\geq 2$ and $U$ is Gorenstein. Then 
  $$
  \omega_X\leteq \iota_*\omega_U
  $$
  is called the \emph{canonical sheaf} of $X$. 
\end{skdefini}

\begin{sklem}
  If $X$ admits a dualizing complex, using the above definition for $\omega_X$, one
  still has that $$\omega_X\simeq h^{-d}(\omega_X^\mydot),$$ where $d=\dim X$. In
  particular, the two definitions of the dualizing sheaf agree.
\end{sklem}

\begin{proof}
  This follows from \eqref{eq:canonical}.
\end{proof}

\begin{skrem}
  We are now in a perfect position to take a deep breath, make a few observations,
  and lose any inhibition we might have against working with non-normal varieties.
  Being normal is the same as being $R_1$ and $S_2$ and we are replacing that with
  being $G_1$ and $S_2$. In other words, we are not going wild with all kinds of
  weird schemes. As far as our canonical divisors are concerned we are not much worse
  off than working with normal varieties. The main thing to keep in mind is that our
  varieties may be singular along a divisor. This means that for example one has to
  be careful when working with Weil divisors. However, the extent of this is
  essentially that by the $G_1$ assumption $\omega_X$ is a line bundle near the
  general points of the $1$-codimensional part of the singular locus of $X$ and hence
  we may choose canonical divisors whose support does not contain any components of
  that $1$-codimensional singular locus. This implies that $X$ is non-singular at the
  general points of these canonical divisors, so we may work with them as we are used
  to work with Weil divisors.  In addition, we will put even more restrictions on our
  singularities. In particular, our stable varieties will only have double normal
  crosssings in codimension $1$. These are arguably the simplest non-normal
  singularities and they are also Gorenstein.
\end{skrem}

As indicated at the beginning of this section, in order to construct our moduli
spaces one needs a canonical polarization on our stable varieties. The obvious
assumption would be to require that stable varieties are Gorenstein. This works in
dimension $1$, but not in higher dimension. Consider a cone over a quartic rational
scroll in $\bP^5$. Then a general pencil of hyperplanes defines a family of smooth
varieties degenerating to one that is not Gorenstein; a cone over a quartic rational
curve in $\bP^4$. For a more detailed explanation of this example see \ref{app:A}.
Taking a branched cover over a general high degree hypersurface section of the cone
one obtains a family of smooth canonically polarized varieties degenarating to one
with the same kind of singularities as above. This example shows that if one sticks
to Gorenstein singularities, or even just to those for which $\omega_X$ is a line
bundle, one will not get a compact moduli space.

So, if $\omega_X$ is not a line bundle, how does one get a ``canonical
polarization''?  The point is that even though one cannot assume that $\omega_X$ is a
line bundle,  may assume that some power of it is. Of course, since $\omega_X$ is
not a line bundle, one has to be careful what ``power'' means. Tensor powers of
non-locally free sheaves tend to get even worse. For instance, tensor powers of
torsion-free, or even reflexive sheaves may have torsion or co-torsion.  Also, we
want the power to be still asssociated to a Weil divisor. In other words, we want it
to be a reflexive sheaf, i.e., we need to take {reflexive powers}:

\begin{skdefini} 
  Let $X$ be a scheme that admits a canonical sheaf $\omega_X$. (For instance it
  admits a dualizing complex or it is $G_1$ and $S_2$).  Then one defines the
  \emph{pluricanonical sheaves} of $X$ as the reflexive powers of the canonical sheaf
  of $X$:
   $$
   \omega_X^{[m]}\leteq \left(\omega_X^{\otimes m}\right)^{**}.
   $$
\end{skdefini}

\begin{sklem}\label{lem:can=dual}
  Let $X$ be a scheme that is $G_1$ and $S_2$.  Then for any $m\in \bZ$,
  $$
  \omega_X^{[m]}\simeq \sO_X(mK_X).
  $$
\end{sklem}

\begin{proof}
  Let $\iota:U\into X$ be an open dense subset of $X$ such that $\codim_X(X\setminus
  U)\geq 2$ and ${\omega_X}\resto U\simeq \omega_U$ is a line bundle. It follows that
  ${\omega_X^{\otimes m}}\resto U\simeq
  \omega^{\otimes m}_U$ is a line bundle, and hence
  $$
  {\omega_X^{[m]}}\resto U\simeq 
  \sO_X(mK_X)\resto U.
  $$
  Since both $\omega_X^{[m]}$ and $\sO_X(mK_X)$ are reflexive, this means that they
  are isomorphic cf.\ \cite[1.11]{MR1291023}. 
\end{proof}

This means that if $X$ is $G_1$ and $S_2$, then one may work with pluricanonical
divisors the same way as if $X$ was normal.

\begin{skrem}
  Talking about Weil divisors on non-normal schemes is tricky, because in order to
  define the multiplicity of a function along a prime divisor and hence define the
  notion of linear equivalence of Weil divisors, one needs the local rings of general
  points of these prime divisors to be DVRs. Therefore, one only considers prime
  divisors that are not contained in the singular locus of the ambient scheme. The
  condition $G_1$ ensures that the canonical sheaf may be represented by a Weil
  divisor that satisfies this requirement.
\end{skrem}

We are now ready to introduce the notion that allows us to have canonical
polarizations even if $\omega_X$ is not a line bundle.

\begin{skdefini}\label{def:Q-Cartier}
  Let $X$ be a scheme that admits a canonical sheaf $\omega_X$. Then, as in
  \eqref{def:main-defs}, $\omega_X$ is called a \emph{$\bQ$-line bundle} if some
  pluricanonical sheaf $\omega_X^{[m]}$ is a line bundle.
\end{skdefini}

\noindent
As a direct consequence of \eqref{lem:can=dual} one obtains:

\begin{sklem}\label{lem:dual-Q-Cartier}
  Let $X$ be a scheme that is $G_1$ and $S_2$. Then $K_X$ is $\bQ$-Cartier if and
  only $\omega_X$ is a $\bQ$-line bundle. 
\end{sklem}

\section{Singularities of the minimal model program}\label{sec:sing-mmp}

It is time to take a more detailed look at the singularities we have encountered and
give precise definitions.  For an excellent introduction to this topic the reader is
urged to take a thorough look at Miles Reid's {\it Young person's guide to canonical
  singularities} \cite{Reid87}. For the precise theory the standard reference is
\cite{KM98} and for recent results one may consult \cite{Hacon-Kovacs10}.

\subsection{Log canonical singularities}

As we have already seen in the case of stable curves, in order to construct compact
moduli spaces one must deal with non-normal singularities as that is the nature of
degenerations: normalization does not work in families.  However, as a warm-up, let
us first define the normal and more traditional singularities that are relevant in
the minimal model program. This will help understanding the somewhat more technical
definitions required to deal with the non-normal case.

\begin{skdefini}\label{def:lc-sing}
  Let $X$ be a normal variety such that $K_X$ is $\bQ$-Cartier and $\phi:\wtilde X\to
  X$ a resolution of singularities with a normal crossing exceptional divisor
  $E=\union E_i$.  One would like to compare the canonical divisors of $\wt X$ and
  $X$.  Since $\phi$ is an isomorphism on an open set this means that the relative
  canonical divisor, that is, the difference between $K_{\wt X}$ and the pull-back of
  $K_X$ is a divisor supported entirely on the exceptional locus. However, as $K_X$
  is not necesssarily Cartier one may not be able to pull it back. One may pull back
  a multiple of it, so one compares that to the same multiple of $K_{\wt X}$. Then
  one divides the difference by the appropriate power. Notice that this way one may
  actually define the pull-back of $K_X$ as a $\bQ$-divisor:
  $$
  \phi^*K_X\leteq \frac 1m \phi^*(mK_X),
  $$ 
  where $m$ is such that $mK_X$ is Cartier. Then one may indeed compare the canoncial
  divisors of $\wt X$ and $X$:
  $$
  K_{\wtilde X} \sim_{\bQ} \phi^*K_X + \sum a_i E_i.
  $$
  where $a_i\in \bQ$.  
Then
  \begin{center}
    $X$ has \quad $\begin{matrix}
      \text{\emph{terminal}}\\
      \text{\emph{canonical}}\\
      \text{\emph{log terminal}}\\ 
      \text{\emph{log canonical}}
    \end{matrix}$ \quad singularities
    if\quad $\begin{matrix} a_i>0. \\ a_i \geq 0. \\
      a_i> -1.\\ a_i\geq -1.\end{matrix}$
  \end{center}
  for all $i$ and any resolution $\phi$ as above.
\end{skdefini}

\begin{skrem}
  We saw in \S\ref{sec:stable-singularities} that the ``right'' class of
  singularities for the total space of a stable family is that of \emph{canonical
    singularities} and that this leads to the fibers having \emph{log canonical
    singularities}. Here we extended the definition of canonical singularities from
  varieties whose canonical sheaf is a line bundle to those whose canonical sheaf is
  a $\bQ$-line bundle. We will generalize these definitions to include the non-normal
  relatives of these singularities in \S\S\ref{sec:semi-log-canonical} which will be
  the right class for ``stable singularities''.  

  Next we will see further evidence supporting this claim.
\end{skrem}

\begin{skexample}\label{ex:cone}
  This is an auxiliary example that I will use later. 

  \noindent
  Let $\Xi=(x^d+y^d+z^d+tw^d=0)\subseteq \bP^3_{x:y:z:w}\times
  \bA^1_t$. The special fiber $\Xi_0$ is a cone over a smooth plane
  curve of degree $d$ and the general fiber $\Xi_t$, for $t\neq 0$, is
  a smooth surface of degree $d$ in $\bP^3$.
\end{skexample}

\begin{skFact}\label{old-eq:2}
  Let $W$ be a smooth variety and $X=X_1\cup X_2\subseteq W$ such that $X_1$ and
  $X_2$ are Cartier divisors in $W$. Then by adjunction 
  \begin{align*}
    K_X &\sim (K_W +X_1+X_2)\resto {X} \\
    K_{X_1} &\sim (K_W +X_1)\resto {X_1} \\
    K_{X_2} &\sim (K_W +X_2)\resto {X_2}, 
  \end{align*}
  and hence 
  \begin{align*}
    K_X\resto{X_1} &\sim K_{X_1}+X_{2}\resto{X_1}\\
    K_X\resto{X_2} &\sim K_{X_2}+X_{1}\resto{X_2}.
  \end{align*}
  It turns out that these latter equalitites are true under more general conditions
  and hence they allow one to check when the canonical divisor of a reducible variety
  is ample by working with the canonical divisor of its irreducible components.
\end{skFact}

\begin{skexample}\label{ex:stable-needs-lc}
  As before, let $f:X\to B$ be a flat morphism, $B$ a smooth curve, and $\phi:\wt
  X\to X$ a resolution of singularites. In this example assume that $X_0$ is a normal
  projective surface with $K_{X_0}$ ample and an isolated singular point $P\in\Sing
  X_0$ such that $X_0$ is isomorphic to a cone $\Xi_0\subseteq \bP^3$ as in
  Example~\ref{ex:cone} locally analytically near $P$.  Assume further that $X$ is
  smooth. One would like to see whether one may resolve the singular point $P\in X_0$
  and still stay within our moduli problem, i.e., that $K$ would remain ample. For
  this purpose one may assume that $P$ is the only singular point of $X_0$.

  Because of the assumption on the singularities one may assume that $\phi$ is the
  blowing up of $P\in X$ and let $\what X_0$ denote the strict transform of $X_0$ on
  $\wt X$.  Then $\wt X_0=\what X_0\cup E$ where $E\simeq \bP^2$ is the exceptional
  divisor of the blow up.  Clearly, $\phi:\what X_0\to X_0$ is the blow up of $P$ on
  $X_0$, so it is a smooth surface and $\what X_0\cap E$ is isomorphic to the degree
  $d$ curve over which $X$ is locally analytically a cone.

  One would like to determine the condition on $d$ that ensures that the canonical
  divisor of $\wt X_0$ is still ample. According to (\ref{old-eq:2}) this means that
  one needs that $K_E+\what X_0\resto E$ and $K_{\what X_0}+E\resto{\what X_0}$ be
  ample.  As $E\simeq \bP^2$, $\omega_E\simeq \sO_{\bP^2}(-3)$, so $\sO_E(K_E+\what
  X_0\resto E)\simeq \sO_{\bP^2}(d-3)$. This is ample if and only if $d>3$.

  As this computation is local near $P$ the only relevant issue about the ampleness
  of $K_{\what X_0}+E\resto{\what X_0}$ is whether it is ample in a neighbourhood of
  $E_0\leteq E\resto{\what X_0}$. By \eqref{claim:surface-ampleness} this is equivalent
  to asking when $(K_{\what X_0}+E_0)\cdot E_0$ is positive.
  
  \begin{skclaim}\label{claim:surface-ampleness}
    Let $Z$ be a smooth projective surface with non-negative Kodaira dimension
    and $\Gamma\subset Z$ an effective divisor. If $(K_Z+\Gamma)\cdot C>0$ for every
    proper curve $C\subset Z$, then $K_Z+\Gamma$ is ample.
  \end{skclaim}

  \begin{proof}
    By the assumption on the Kodaira dimension there exists an $m>0$ such that $mK_Z$
    is effective, hence so is $m(K_Z+\Gamma)$. Then by the assumption on the
    intersection number, $(K_Z+\Gamma)^2>0$, so the statement follows by the
    Nakai-Moishezon criterium.
  \end{proof}

  Now, observe that by the adjunction formula $(K_{\what X_0}+E_0)\cdot E_0=\deg
  K_{E_0} = d(d-3)$ as $E_0$ is isomorphic to a plane curve of degree $d$.
  Again, one obtains the same condition as above and thus conclude that $K_{\wt X_0}$
  is ample if and only if $d>3$.

  Since the objects that one considers in the current moduli problem must have an
  ample canonical class, one may only replace $X_0$ by $\wt X_0$ if $d>3$.
  For our moduli problem this means that one has to allow cone singularities over
  curves of degree $d\leq 3$. The singularity one obtains for $d=2$ is a rational
  double point, but the singularity for $d=3$ is not, it is not even rational.

  In fact, the above calculation tells us more. One has that $K_{\what
    X_0}=\phi^*K_{X_0}+aE_0$ for some $a\in \bZ$. To compute $a$, first recall that
  $\deg K_{E_0}=d(d-3)$ and $E_0^2=-d$. Then 
  $$
  \deg K_{E_0}=(K_{\what X_0}+E_0)\cdot {E_0}=(\phi^*K_{X_0}+(a+1)E_0)\cdot {E_0} =
  (a+1) E_0^2 = - (a+1) d.
  $$
  Therefore $a=2-d$. In other words, the condition obtained above, that one needs to
  allow cone singularities over plane curves of degree $d\leq 3$ is equivalent to
  allowing log canonical singularities cf.\ \eqref{def:lc-sing}.
\end{skexample}

I have mentioned that stable singularities are not necessarily Cohen-Macau\-lay. Until
we identified the actual class we want to call stable this was more or less an empty
statement. By now, it is rather clear that log canonical singularities will belong to
the class we are looking for, so we might as well point to an example of non-CM log
canonical singularities.

\begin{skexample}\label{ex:lc-not-CM}
  Let $X$ be a cone over an abelian variety of dimension at least $2$. Then $X$ is
  log canonical, but not Cohen-Macaulay. 
\end{skexample}

As mentioned several times, one also has to deal with some non-normal singularities
and in fact in the example in \eqref{ex:stable-needs-lc} one does not really need
that $X$ be normal. In the next few subsections we will see examples of non-normal
singularities that one has to handle. In particular, we will see that one has to
allow the non-normal cousins of log canonical singularities. These are called
\emph{semi-log canonical} singularities and the reader can find their definition in
\eqref{sec:semi-log-canonical}.

\subsection{Normal crossings}

A \emph{normal crossing} singularity is one that is locally analytically (or
formally) isomorphic to the intersection of coordinate hyperplanes in a linear space.
In other words, it is a singularity locally analytically defined as $(x_1x_2\cdots
x_r=0)\subseteq\bA^n$ for some $r\leq n$. In particular, as opposed to the curve
case, for surfaces it allows for triple intersections.  However, triple (or higher)
intersections may be ``semi-resolved'': Let $X=(xyz=0)\subseteq\bA^3$.  Blow up the
origin $O\in\bA^3$, $\sigma:{\rm Bl}_O\bA^3\to \bA^3$ and consider the strict
transform of $X$, $\sigma: \wtilde X\to X$. Observe that $\wtilde X$ has only double
normal crossings and the morphism $\sigma$ is an isomorphism over $X\setminus\{O\}$.
Therefore, this is a {semi-resolution} as defined in
(\ref{def:semi-log-canonical}.\ref{item:semi-resolution}). Double normal crossings
cannot be resolved the same way, because the double locus is of codimension $1$, so
any morphism from any space with any kind of singularities that are not double normal
crossings would fail to be an isomorphism in codimension $1$.

Since normal crossings are (analytically) locally defined by a single equation, they
are {Gorenstein} and hence the canonical sheaf $\omega_X$ is still a line
bundle and so it makes sense to require it to be ample.

These singularities already appear for stable curves, so it is not surprising that
they are still here. As one wants to understand degenerations of one's preferred
families, one has to allow (at least) normal crossings.

Another important point to remember about normal crossings is that they are {\it not}
normal.  For some interesting and perhaps surprising examples of surfaces with normal
crossings see \cite{0705.0926v2}.

\subsection{Pinch points}

Another non-normal singularity that can occur as the limit of smooth varieties is the
\emph{pinch point}. It is locally analytically defined as
$(x_1^2=x_2 x_3^2)\subseteq\bA^n$ ($n\geq 3$).  This singularity is a double normal
crossing away from the {pinch point}.  Its normalization is smooth, but blowing
up the pinch point does not make it any better as shown by the example that follows.

\begin{skexample}
  Let $X=(x_1^2=x_2^2 x_3)\subseteq\bA^3$, where $x_1,x_2,x_3$ are linear coordinates
  on $\bA^3$, $O=(0,0,0)$ and compute ${\rm Bl}_OX$. First, recall that
  $$
  {\rm Bl}_O\bA^3=\{(x_1,x_2,x_3)\times [y_1:y_2:y_3] 
  \vert x_iy_j=x_jy_i \text{ for } i,j=1,2,3 \} \subset \bA^3\times \bP^2,
  $$
  where $y_1,y_2,y_3$ are homogenous coordinates on $\bP^2$. 

  \begin{enumerate}
  \item Assume that $y_1=1$.  Then $x_2=x_1y_2$ and $x_3=x_1y_3$ and the equation of
    the preimage of $X$ becomes $x_1^2=x_1^3y_2^2y_3$. This breaks up into $x_1^2=0$
    and $1=x_1y_2^2y_3$.  The former equation defines the exceptional divisor and the
    latter defines the strict transform of $X$, i.e., ${\rm Bl}_OX$. This does not
    have any points over $O\in X$, so on this chart, the blow up morphism ${\rm
      Bl}_OX\to X$ is an isomorphism and ${\rm Bl}_OX$ is smooth.
  \item Assume that $y_2=1$.  Then $x_1=x_2y_1$ and $x_3=x_2y_3$ and the equation of
    the preimage of $X$ becomes $x_2^2y_1^2=x_2^3y_3$. This breaks up into $x_2^2=0$
    and $y_1^2=x_2y_3$. Again, the former equation defines the exceptional divisor
    and the latter the strict transform of $X$, ${\rm Bl}_OX$. Notice that on this
    chart a coordinate system is given by $x_2, y_1, y_3$ and the equation defines a
    quadric cone. Then blowing up the vertex of the cone gives a resolution on this
    chart.
  \item Assume that $y_3=1$.  Then $x_1=x_3y_1$ and $x_2=x_3y_2$ and the equation of
    the preimage of $X$ becomes $x_3^2y_1^2=x_3^3y_2^2$. This breaks up as $x_3^2=0$
    and $y_1^2=y_2^2x_3$.  Again, the former equation defines the exceptional divisor
    and the latter the strict transform of $X$, ${\rm Bl}_OX$. Notice that on this
    chart a coordinate system is given by $x_3, y_1, y_3$ and the latter equation is
    the same as the one we started with. So, ${\rm Bl}_OX$ again has a pinch point.
  \end{enumerate}
  This computation shows that the blow-up of a pinch point will be, if anything, more
  singular, than the original and at best it can be resolved to be a pinch point
  again. 
\end{skexample}

From this example one concludes that a pinch point cannot be resolved or even just
made somewhat ``better'' by only trying to change it over the pinch point. It may
only be resolved by taking the normalization. As in the case of double normal
crossings, this is not an isomorphism in codimension $1$.

\begin{demo}{Observation}\label{obs:pinch-points-and-double-nc}
  Double normal crossings and pinch points share the following interesting
  properties:
  \begin{enumerate}
  \item Their normalization is smooth.
  \item The normalization morphism is \emph{not} an isomorphism in codimension $1$.
  \item It is not possible to find a partial resolution that is an isomorphism in
    codimension $1$ that would make them better in any reasonable sense.
  \end{enumerate}
\end{demo}

\begin{skrem}
  Notice that all normal crossings share the first two properties, but, in dimension
  at least $2$, not the third one as they may be partially resolved to double normal
  crossings.

  One concludes that double normal crossing and pintch point singularities are
  unavoidable.  However, at the same time, they should be viewed as the simplest
  non-normal singularities. In fact, in some sense they are much simpler than most
  normal singularities.

  Furthermore, all other singularities can be resolved to these: Any reduced scheme
  admits a {partial resolution} to a scheme with only double normal crossings and
  pinch points such that the resolution morphism is an isomorphism wherever the
  original scheme is smooth, or has only double normal crossings or pinch points
  \cite{kollar-semi-log-resolutions}. Of course, this only gives a partial resolution
  that is an isomorphism in codimension $1$ if the scheme one starts with has double
  normal crossings in codimension $1$ already. However, this turns out to be a
  condition one can achieve.

  We will discuss relevant partial resolutions in more detail in
  \eqref{def:semi-log-canonical}.
\end{skrem}

\subsection{Semi-log canonical singularities}
\label{sec:semi-log-canonical}

Next, I will make the definition of the non-normal version of log canonical
singularities precise.

\begin{skdefini}\label{def:semi-log-canonical}
  Let $X$ be a scheme of dimension $n$ and $x\in X$ a closed point.
  \begin{enumerate}
  \item $x\in X$ is a \emph{double normal crossing} if it is locally analytically (or
    formally) isomorphic to the singularity
    $$
    \left\{ 0\in (x_0x_1=0) \right\} \subseteq \left\{0\in \bA^{n+1}
    \right\},
    $$
    where $n\geq 1$.
  \item $x\in X$ is a \emph{pinch point} if it is locally analytically
    (or formally) isomorphic to the singularity
    $$
    \left\{ 0\in (x_0^2=x_1^2x_2) \right\} \subseteq \left\{0\in \bA^{n+1} \right\},
    $$
    where $n\geq 2$.
  \item $X$ is \emph{semi-smooth} if all closed points of $X$ are either smooth, or a
    double normal crossing, or a pinch point. In this case, unless $X$ is smooth,
    $D_X\leteq\Sing X\subseteq X$ is a smooth $(n-1)$-fold. If $\nu:\wtilde X\to X$
    is the normalization, then $\wtilde X$ is smooth and $\wtilde D_X\leteq
    \nu^{-1}(D_X)\to D_X$ is a double cover ramified along the pinch locus.
    Furthermore, the definition implies that if $X$ is semi-smooth, then it is
    Gorenstein. In particular, it admits a canonical sheaf $\omega_X$ which is a line
    bundle.
  \item \label{item:semi-resolution} A morphism, $\phi:Y\to X$ is a
    \emph{semi-resolution} if
    \begin{itemize}
    \item $\phi$ is proper,
    \item $Y$ is semi-smooth,
    \item no component of $D_Y$ is $\phi$-exceptional, and
    \item there exists a closed subset $Z\subseteq X$, with
      $\codim(Z,X)\geq 2$ such that 
      $$
      \phi\resto{\phi^{-1}(X\setminus Z)} : \phi^{-1}(X\setminus Z)
      \overset\simeq\to X\setminus Z
      $$
      is an isomorphism.
    \end{itemize}
    Let $E$ denote the exceptional divisor (i.e., the codimension 1
    part of the exceptional set, not necessarily the whole exceptional
    set) of $\phi$. Then $\phi$ is a \emph{good semi-resolution} if $E\union
    D_Y$ is a divisor with global normal crossings on $Y$.
  \item \label{def:slc} $X$ has \emph{semi-log canonical} (\emph{slc}) (resp.\
    \emph{semi-log terminal} (\emph{slt})) singularities if
    \begin{enumerate}
    \item $X$ is reduced,
    \item $X$ is $S_2$,
    \item $X$ admits a canonical sheaf $\omega_X$, which is a $\bQ$-line bundle of
      index $m$, and
    \item\label{item:9} there exists a good semi-resolution of singularities
      $\phi:\wtilde X\to X$ with exceptional divisor $E=\union E_i$ such that
      $\omega_{\wt X}^{m}\simeq \phi^*\omega_X^{[m]}\otimes \sO_{\wt X}(m\cdot\sum  a_i
      E_i)$ with $a_i\in \bQ$ and $a_i\geq -1$ (resp.\ $a_i > -1$) for all $i$.
    \end{enumerate}
  \end{enumerate}
\end{skdefini}

\begin{skrem}\label{rem:slc}
  A semi-smooth scheme has at worst hypersurface singularities, so in particular it
  is Gorenstein.  This means that condition
  (\ref{def:semi-log-canonical}.\ref{item:9}) implies that $X$ is $G_1$.  In other
  words it follows that $X$ admits a canonical sheaf. However,
  (\ref{def:semi-log-canonical}.\ref{item:9}) cannot be stated without assuming this
  first. On the other hand, it means that one may assume that $X$ is $G_1$ instead. In
  other words, without loss of generality one may define slc (resp.\ slt)
  singularities as those satisfying that
  \begin{enumerate}
  \item $X$ is reduced,
  \item $X$ is $G_1$ and $S_2$,
  \item $\omega_X$ is a $\bQ$-line bundle of index $m$, and
  \item there exists a good semi-resolution of singularities
    $\phi:\wtilde X\to X$ with exceptional divisor $E=\union E_i$ such that
    $\omega_{\wt X}^{m}\simeq \phi^*\omega_X^{[m]}\otimes \sO_{\wt X}(m\cdot\sum a_i
    E_i)$ with $a_i\in \bQ$ and $a_i\geq -1$ (resp.\ $a_i > -1$) for all $i$.
  \end{enumerate}
\end{skrem}
%
\begin{newnumrp}{}
  Furthermore, once one assumes that $X$ is $G_1$ and $S_2$, one may work with
  canonical divisors instead of canonical sheaves. In other words, we may also define
  slc (resp.\ slt) singularities as those satisfying that
  \begin{enumerate}
  \item $X$ is reduced,
  \item $X$ is $G_1$ and  $S_2$,
  \item $K_X$ is $\bQ$-Cartier, and 
  \item\label{item:10} there exists a good semi-resolution of singularities
    $\phi:\wtilde X\to X$ with exceptional divisor $E=\union E_i$ such that
    $K_{\wtilde X} \equiv \phi^*K_X + \sum a_i E_i$ with $a_i\in \bQ$ and $a_i\geq
    -1$ (resp.\ $a_i > -1$) for all $i$.
  \end{enumerate}
  Again, (\ref{def:semi-log-canonical}.\ref{item:10}) implies that $X$ is $G_1$, but
  one needs that assumption even to work with $K_X$. Of course, instead of $G_1$, one
  may start by assuming that $X$ admits a semi-resolution, then conclude that
  canonical divisors may be defined and then go on with the definition.
\end{newnumrp}

\begin{skrem}
  It is relatively easy to prove that if $X$ has semi-log canonical (resp.\ semi-log
  terminal) singularities, then the condition in
  (\ref{def:semi-log-canonical}.\ref{item:9}) follows for \emph{all} good
  semi-resolutions.
\end{skrem}

\begin{skrem}
  One may further generalize the notion of semi log canonical and define \emph{weakly
    semi log canonical} singularities as those that are seminormal, $S_2$ and with an
  appropriately chosen divisor on the normalization, that pair is log canonical. In
  this context semi-log canonical singularities are exactly those weakly semi-log
  canonical divisors that are $G_1$.
  For the precise definition and more details on these singularities and their
  relationships see \cite{KSS10}.
\end{skrem}

\begin{skrem}
  In the definition of a semi-resolution, one could choose to require that the
  exceptional set be a divisor. This leads to slightly different notions. It is still
  to be seen whether this variation leads to anything interesting (that is, anything
  interesting that is different from all the interesting things the definition above
  leads to).  For more on singularities related to semi-resolutions see \cite{KSB88},
  \cite{K+92}, and \cite{kollar-semi-log-resolutions}.
\end{skrem}

\medskip

\noindent
Now we are ready to define stable varieties in arbitrary dimensions.

\begin{skdefini}\label{def:stable}
  A variety $X$ is called \emph{stable} if  
  \begin{enumerate}
  \item $X$ is projective,
  \item $X$ has semi log canonical singularities, and
  \item $\omega_X$ is an ample $\bQ$-line bundle.
  \end{enumerate}
\end{skdefini}

\begin{skrem}
  Notice that if $\dim X=1$, then this is equivalent with the previous definition of
  a stable curve \eqref{def:stable-curves}.
\end{skrem}

We should also revisit the definition of \emph{stable families}. As opposed to the
case of curves, our stable varieties are canonically polarized by a $\bQ$-line bundle
and not a line bundle. As far as embedding into a projective space, computing
intersection numbers, and pulling back pluricanonical sheaves are concerned this does
not cause a big difference. However it introduces an additional element to which one
has to pay attention when dealing with families.

We do not simply want a family of canonically polarized varieties but a family where
these canonical polarizations are compatible. In other words, we want a relative
canonical polarization of the family that restricts to the canonical polarization of
the members of the family. In particular, we want that for a stable family $X\to B$,
\begin{equation}
  \label{eq:9}
  \omega_{X/B}\resto{X_b}\simeq \omega_{X_b}\quad\text{for all $b\in B$.}
\end{equation}

It turns out that for curves this follows from the other assumptions and as a matter
of fact we have also (secretly) assumed it during our quest for stable varieties cf.\
\eqref{obs:total-space}. 

The only point to keep in mind is that now if one wants to define stable families
only using properties of the fibers, as in the case of curves, then one might lose
this condition accidentally. For an example that this can actually happen, that is,
that there exists families of stable varieties that are not stable families in the
sense of our earlier requirements see \ref{app:A}.

\begin{skdefini}\label{def:weakly-stable-family}
  A morphism $f:X\to B$ is called a \emph{weakly stable family} if it satisfies the
  following conditions:
  \begin{enumerate}
  \item $f$ is flat and projective
  \item $\omega_{X/B}$ is a relatively ample $\bQ$-line bundle
  \item $X_b$ has semi log canonical singularities for all $b\in B$.
  \end{enumerate}
\end{skdefini}

This definition actually still hides one very important detail. The fact that
$\omega_{X/B}$ is a $\bQ$-line bundle means that it has an index, that is, an integer
$N\in\bN$ such that $\omega_{X/B}^{[N]}$ is a line bundle and this is the smallest
positive reflexive power of $\omega_{X/B}$ which is a line bundle. It follows that
then (cf.\ \cite[2.6]{Hassett-Kovacs04}),
\begin{equation}
  \label{seq:reflexive-restriction}
  \omega_{X/B}^{[N]}\resto{X_b} \simeq \omega_{X_b}^{[N]}.
\end{equation}

\noindent
In particular, $\omega_{X_b}$ is a $\bQ$-line bundle of index $m$ for some $m$ that
divides $N$. This means that $X_b$ may appear in weakly stable families whose
relative canonical sheaf is a $\bQ$-line bundle of index $N$ for any multiple of $m$.
This actually leads to a problem with respect to the moduli spaces of these families.
There may be weakly stable families all of whose members have canonical sheaves of
index $m$, but the relative canonical sheaf of the family has index $N>m$. In other
words one might encounter families that are admissible as families of varieties of
index $N$ but not as families of varieties of index $m$, even though all members have
index $m$. A reasonable resolution of this problem is to ask that besides
\eqref{seq:reflexive-restriction} a similar restriction should hold for all reflexive
powers of the relative canonical sheaf.

\begin{skdefini}\label{def:stable-family}
  A weakly stable family $f:X\to B$ is called a \emph{stable family} if it satisfies
  \emph{Koll\'ar's condition}, that is, for any $m\in\bN$
\begin{equation*}
  \omega_{X/B}^{[m]}\resto{X_b} \simeq \omega_{X_b}^{[m]}.
\end{equation*}
\end{skdefini}

\begin{skrem}
  Notice that it is always true that the double dual of the restriction of the
  relative pluricanonical sheaf is the corresponding pluricanonical sheaf of the
  fiber: 
  \begin{equation*}
    \left(\omega_{X/B}^{[m]}\resto{X_b}\right)^{**} \simeq \omega_{X_b}^{[m]},
  \end{equation*}
  so the main content of Koll\'ar's condition is that the restriction of all
  pluricanonical sheaves have to be reflexive. For more on the definition of stable
  families and the corresponding moduli functors see \cite[\S7]{MR2483953}.
\end{skrem}

\begin{skrem}
  Notice further that Koll\'ar's condition includes condition \eqref{eq:9}.
  Interestingly, it is not obvious that even this simple condition holds for weakly
  stable families. It holds for families of curves since stable curves are
  Gorenstein. It also holds for families of surfaces since stable surfaces are
  Cohen-Macaulay on account of being $S_2$ and this condition holds for families of
  Cohen-Macaulay varieties cf. \cite[3.5.1]{Conrad00}.

  However, stable varieties of dimension $\geq 3$ are not necessarily
  Cohen-Macau\-lay \eqref{ex:lc-not-CM}, so it is absolutely not obvious weather the
  relative canonical sheaf is invariant under base change. It turns out that this is
  actually true by \eqref{cor:base-change-for-omega} cf.\ \cite{KK10}. To see that
  this invariance under base change for weakly stable families is highly non-trivial
  the reader is referred to the examples in \cite{patakfalvi-base-change} that show
  that this statement is sharp in some reasonable sense.
\end{skrem}

\section{Duality and vanishing}\label{sec:duality-vanishing}

In this section I will first state two fundamental theorems that will be used later
and then list a few vanishing theorems that are important in both the minimal model
program and higher dimensional moduli theory.

Before anything else, we need a few definitions.

\begin{skdefini}\label{def:derived}
  Let $X$ be a complex scheme (i.e., a scheme of finite type over $\bC$) of dimension
  n. Let $D_{\rm filt}(X)$ denote the derived category of filtered complexes of
  $\sO_{X}$-modules with differentials of order $\leq 1$ and $D_{\rm filt, coh}(X)$
  the subcategory of $D_{\rm filt}(X)$ of complexes $\sfK$, such that for all $i$,
  the cohomology sheaves of $Gr^{i}_{\rm filt}\sfK$ are coherent cf.\
  \cite{DuBois81}, \cite{GNPP88}.  Let $D(X)$ and $D_{\rm coh}(X)$ denote the derived
  categories with the same definition except that the complexes are assumed to have
  the trivial filtration.  The superscripts $+, -, b$ carry the usual meaning
  (bounded below, bounded above, bounded).  Isomorphism in these categories is
  denoted by $\qis$.  A sheaf $\sF$ is also considered as a complex $\sF^\mydot$ with
  $\sF^0=\sF$ and $\sF^i=0$ for $i\neq 0$.  If $\sfK$ is a complex in any of the
  above categories, then $h^i(\sfK)$ denotes the $i$-th cohomology sheaf of $\sfK$.

  The right derived functor of an additive functor $F$, if it exists, is denoted by
  $\myR F$ and $\myR^iF$ is short for $h^i\circ \myR F$. Furthermore, $\bH^i$
  will denote $\myR^i\Gamma$,
  where $\Gamma$ is the functor of global sections.
  Note that according to this terminology, if $\phi\col Y\to X$ is a morphism and
  $\sF$ is a coherent sheaf on $Y$, then $\myR\phi_*\sF$ is the complex whose
  cohomology sheaves give rise to the usual higher direct images of $\sF$.

  Similarly, the left derived functor of an additive functor $F$, if it exists, is
  denoted by $\myL F$ and $\myL^iF$ is short for $h^i\circ \myL F$.
\end{skdefini}

The next two theorems are very important in studying cohomological properties of
singular varieties. 

\begin{skthm}
  [\protect{(Grothendieck Duality) \cite[VII]{MR0222093}}]
  \label{thm:GD}
  Let $\phi: Y\to X$ be a proper morphism between finite dimensional noetherian
  schemes that admit dualizing complexes. Then for any bounded complex $\sfG\in
  D^b(Y)$,
  $$
  \myR\phi_*\myR\sHom_Y(\sfG, \omega_Y^\mydot)\qis \myR\sHom_X(\myR\phi_* \sfG,
  \omega_X^\mydot).
  $$
\end{skthm}

\medskip

\begin{skthm}
  [\protect{(Adjointness of $\phi_*$ and $\phi^*$) \cite[II.5.10]{MR0222093}}] 
  \label{thm:adjoint}
  Let $\phi: Y\to X$ be a proper morphism. Then for any bounded complexes $\sfF\in
  D^b(X)$ and $\sfG\in D^b(Y)$,
  $$
  \myR\phi_*\myR\sHom_Y(\myL\phi^*\sfF, \sfG)\qis \myR\sHom_X(\sfF,\myR\phi_*\sfG). 
  $$
\end{skthm}

Vanishing theorems have played a central role in algebraic geometry for the last
couple of decades, especially in classification theory. Koll\'ar \cite{Kollar87b}
gives an introduction to the basic use of vanishing theorems as well as a survey of
results and applications available at the time. For more recent results one should
consult \cite{MR853449,EV92,Ein97,Kollar97e,KSmith95,Kovacs00c,Kovacs02,Kovacs03c,
  Kovacs03a}.  Because of the availability of those surveys, I will only recall
statements that are important for the present article. Nonetheless, any discussion of
vanishing theorems should start with the fundamental vanishing theorem of Kodaira.

\begin{skthm}[\cite{Kodaira53}]\label{kodaira}
  Let $Y$ be a smooth complex projective variety and $\sL$ an ample line bundle on
  $Y$. Then
  $$
  \coh i.Y.\omega_Y\otimes\sL.=0 \text{ for }i\neq 0.
  $$
\end{skthm}

This has been generalized in several ways, but as noted above I will only state what
I use in this article. For the many other generalizations the reader is invited to
peruse the above references.  

The original statement of Kodaira was generalized to allow semi-ample and big line
bundles in place of ample ones by Grauert and Riemenschneider.

\begin{skthm}[\cite{Gra-Rie70b}]\label{gr}
  Let $Y$ be a smooth complex projective variety and $\sL$ a semi-ample and big line
  bundle on $Y$. Then
  $$
  \coh i.Y.\omega_Y\otimes\sL.=0\text{ for }i\neq 0.
  $$
\end{skthm}

\noindent
This also has a relative version:

\begin{skthm}[\cite{Gra-Rie70b}]\label{GR}
  Let $Y$ be a smooth complex variety, $\phi:Y\to X$ a projective birational
  morphism, and $\sL$ a semi-ample line bundle on $Y$. Then
  $$
  \myR^i\phi_*\left(\omega_Y\otimes\sL\right)=0\text{ for }i\neq 0.
  $$
\end{skthm}

\noindent
By Serre duality both \eqref{kodaira} and \eqref{gr} has a dual version:

\begin{skthm}\label{gr-dual}
  Let $Y$ be a smooth complex projective variety and $\sL$ a semi-ample and big line
  bundle on $Y$. Then
  $$
  \coh j.Y.\sL^{-1}.=0\text{ for }j\neq \dim Y.
  $$
\end{skthm}

What would be the dual version of \eqref{GR} in the same spirit?  Instead of Serre
duality one would have to use Grothendieck duality:

Let $Y$ be a smooth complex variety of dimension $d$, $\phi:Y\to X$ a projective
morphism, and $\sL$ a semi-ample line bundle on $Y$. Then
\begin{multline}
  \label{eq:GR-to-dual}
  \myR\sHom_X(\myR\phi_*(\omega_Y\otimes \sL), \omega_X^\mydot)\qis \\ \qis
  \myR\phi_*\myR\sHom_Y(\omega_Y\otimes \sL, \omega_Y[d])\qis 
  \myR\phi_*\sL^{-1}[d] 
\end{multline}

In the case of \eqref{kodaira} and \eqref{gr} $X=\Spec \bC$, so $\omega_X^\mydot\qis
\bC$. Then the the left hand side is quasi-isomorphic to the dual of
$\myR\phi_*(\omega_X\otimes\sL)\qis \phi_*(\omega_X\otimes\sL)\simeq H^0(Y,
\omega_X\otimes\sL)$. Therefore $h^i(\myR\phi_*\sL^{-1}[d])=0$ for $i\neq 0$. This is
how \eqref{gr-dual} follows: $\myR^j\phi_*\sL^{-1}=\coh j.Y.\sL^{-1}.=0$ for $j\neq
d$.

In the case $\phi$ is birational there is a shift by $d$ on both side so the expected
dual form of this vanishing would be 
\begin{equation}
  \label{eq:GR-dual}
  \myR^j\phi_*\sL^{-1}=0 \text{ for $j\neq 0$.}
\end{equation}

However, this does not always hold. To see this let us consider the simplest
semi-ample line bundle, $\sO_Y$. Then $\myR^i\phi_*\omega_Y=0$ for $i\neq 0$ by
\eqref{GR}, so (\ref{eq:GR-to-dual}) reduces to the following:
\begin{equation}
  \label{eq:GD+GR}
  \myR\sHom_X(\phi_*\omega_Y, \omega_X^\mydot)\qis
  \myR\phi_*\sO_Y[d] 
\end{equation}
Now suppose that $X$ is normal and $\omega_Y\simeq \sO_Y$. Then it follows that if
\eqref{eq:GR-dual} holds for $\sL=\sO_Y$, then $\omega_X^\mydot$ has only one
non-zero cohomology sheaf and hence $X$ is Cohen-Macaulay. In other words, if $X$ is
normal, but not CM and $Y$ has a trivial canonical bundle, then \eqref{eq:GR-dual}
does not hold with $\sL=\sO_Y$ or more generally with $\sL=\phi^*\sM$ for any line
bundle $\sM$ on $X$.

The point is that the dual form of the relative Grauert-Riemenschneider vanishing
theorem is a singularity condition on the target of the morphism in question. Notice
that \eqref{eq:GR-dual} follows from \eqref{eq:GD+GR} for $\sL=\sO_X$ if $X$ is
Cohen-Macaulay and $\phi_*\omega_Y\simeq \omega_X$. It turns out that this defines a
very important class of singularities which is the topic of the next section.

\section{Rational singularities}
\label{sec:rtl}

Rational singularities are among the most important classes of singularities.  The
essence of rational singularities is that their cohomological behavior is very
similar to that of smooth points. For instance, vanishing theorems can be easily
extended to varieties with rational singularities.  Establishing that a certain class
of singularities is rational opens the door to using very powerful tools on varieties
with those singularities.

\begin{skdefini}\label{def:rtl-sing}
  Let $X$ be a normal variety and $\phi :Y \rightarrow X$ a resolution of
  singularities. $X$ is said to have \emph{rational} singularities if
  $\myR^i\phi_*\sO_Y=0$ for all $i>0$, or equivalently if the natural map $\sO_X\to
  \myR\phi_*\sO_Y$ is a quasi-isomorphism.
\end{skdefini}

The notion of \emph{irrational centers} is very closely related. For the definition
and basic properties see \cite{Kovacs11a}.

A very useful property of rational singularities is that they are {Cohen-Macau\-lay}.
In fact, this is part of Kempf's characterization of rational singularities:

\begin{skthm}\cite[p.50]{KKMS73}\label{thm:kempf}
  Let $X$ be a normal variety and $\phi:Y\to X$ a resolution of singularities. Then
  $X$ has rational singularities if and only if $X$ is Cohen-Macau\-lay and
  $\phi_*\omega_{Y}\simeq \omega_X$.
\end{skthm}

\begin{proof} Let $d=\dim X$.
If $X$ has rational singularities, then 
\begin{multline*}
  \omega_X^\mydot\qis \myR\sHom_X(\sO_X,\omega_X^\mydot)\qis
  \myR\sHom_X(\myR\phi_*\sO_Y, \omega_X^\mydot)\qis \\
  \qis\myR\phi_*\myR\sHom_Y(\sO_Y, \omega_Y[d])\qis \myR\phi_*\omega_Y[d] \qis
  \phi_*\omega_Y[d],
\end{multline*}
which implies that $X$ has to be Cohen-Macau\-lay and $\omega_X\simeq \phi_*\omega_Y$. 

Similarly, if $X$ is Cohen-Macau\-lay and $\omega_X\simeq \phi_*\omega_Y$, then
$\omega_X^\mydot\qis \phi_*\omega_Y[d]$ and so
\begin{multline*}
  \myR\phi_*\sO_Y \qis \myR\phi_*\myR\sHom_Y(\omega_Y[d],\omega_Y^\mydot) \qis
  \myR\sHom_X(\myR\phi_*\omega_Y[d], \omega_X^\mydot) \qis \\ \qis 
  \myR\sHom_X(\phi_*\omega_Y[d], \omega_X^\mydot) \qis \myR\sHom_X(\omega_X^\mydot,
  \omega_X^\mydot)\qis \sO_X
\end{multline*}
shows that $X$ has rational singularities.
\end{proof}

A very important fact is that log terminal singularities are rational:

\begin{skthm}[\cite{Elkik81}]\label{c-elkik}
  Let $X$ be a variety with log terminal singularities. Then $X$ has rational
  singularities.
\end{skthm}

This is actually an easy consequence of a characterization theorem that will be
stated later in \eqref{thm:rtl-crit}. The proof will be given in
\eqref{thm:elkik-with-proof} after the necessary notation is introduced in
\S\ref{sec:splitting-principle}.

In particular, canonical singularities are rational and as a corollary one obtains
that the total space of a stable family should have rational singularities.

Now we may repeat the investigation that helped us figure out what kind of
singularities stable varieties should have. Previously we figured that if the total
space has canonical singularities then the fibers should have semi log canonical
singularities. Next we would like to see what it means for the fibers that the total
space of the family has rational singularities.

So, let $f:X\to B$ be a family of reduced varieties such that $B$ is a smooth curve
and $X$ has rational singularities. Let $b\in B$ a fixed point and let $\phi: Y\to X$
be a resolution of singularities such that $\supp(\exc(\phi)\cup \phi^{-1}X_b)$ is a
simple normal crossing divisor. Observe that by assumption and construction
$X_b=f^*b$ is a Cartier divisor and $Y_b=\phi^*X_b$.  Following the spirit of our
assumption on stable families \eqref{obs:total-space} assume that $Y_b$ is reduced,
that is, $Y_b=\phi^{-1}X_b$.  One also has the following commutative diagram of
distinguished triangles:
$$
\xymatrix{%
  \sO_X(-X_b) \ar[r] \ar[d]_{\alpha_1} & \sO_X \ar[r]\ar[d]_{\alpha_2} & \sO_{X_b}
  \ar[r]^-{+1}\ar[d]_{\alpha_3} & \
  \\
  \myR\phi_*\sO_Y(-Y_b) \ar[r] & \myR\phi_*\sO_Y \ar[r] & \myR\phi_*\sO_{Y_b}
  \ar[r]^-{+1} & \ }
$$
Notice that if the horizontal morphisms in this diagram are the usual natural
morphisms, then $\alpha_3$ is uniquely determined by $\alpha_1$ and $\alpha_2$ by
\eqref{lem:h-nought}.

As $X$ has rational singularities, $$\alpha_2: \sO_X\to \myR\phi_*\sO_Y$$ is a
quasi-isomorphism. Since $\sO_Y(-Y_b)\simeq \phi^*\sO_X(-X_b)$ it follows by the
projection formula that
$$\alpha_1=\alpha_2\otimes\id_{\sO_X(-X_b)}: \sO_X(-X_b) \to \myR\phi_*\sO_Y(-Y_b)
\qis \myR\phi_*\sO_Y\otimes\sO_X(-X_b)$$ is also a quasi-isomorphism.  Therefore the
triangulated category version of the $9$-lemma (see \ref{app:B}) implies that
$$\alpha_3: \sO_{X_b}\to \myR\phi_*\sO_{Y_b}$$ is also a quasi-isomorphism. Note that
this does not mean that $X_b$ has rational singularities as $Y_b$ is not a resolution
of singularities, it is in general not even birational to $X_b$. However, it
definitely means that these singularities are not too far from rational
singularities.

We found in \S\ref{sec:stable-singularities} that one cannot expect the fibers to
have the same type of singularities as the total space, just as one cannot expect all
hyperplane sections of varieties in general to have the same type of singularities as
the original varieties. Similarly here one cannot expect to have the members of the
family have rational singularities. However, just as in
\S\ref{sec:stable-singularities}, one finds that the singularities of the fibers are
not too much worse. These are called Du~Bois singularities and we will get acquainted
with them in the next few sections. Notice that the condition we obtained here is
almost identical to the one given by Schwede's criterion in
\eqref{EasyDuBoisCriterion}.

\section{DB singularities}\label{sec:db-singularities}

Du~Bois singularities are probably harder to appreciate than rational singualrites at
first, but they are equally important. Their main importance comes from two facts:
They are not too far from rational singularities, that is, they share many of their
properties, but the class of Du~Bois singularities is more inclusive than that of
rational singularities.  For instance, log canonical singularities are Du~Bois, but
not necessarily rational. 

Du~Bois singularities are defined via Deligne's Hodge theory and so their strong
connection to the singularities of the minimal model program might seem unexpected.
Nevertheless, they play a very important role. We will need a little preparation
before we can define these singularities, but first I would like to mention a few
facts to underline their importance. 

The concept of Du Bois singularities, abbreviated as \emph{DB}, was introduced by
Steenbrink in \cite{Steenbrink83} as a weakening of rationality. The following
statement is a direct consequence of the definition and this is the most important
property of a DB singularity:

\begin{skthm}\label{thm:hodge-for-db}
  Let $X$ be a proper scheme of finite type over $\bC$.  If $X$ has only DB
  singularities, then the natural map
  $$
  H^i(X^{\rm an},\bC)\to H^i(X^{\rm an},\sO_{X^{\rm an}})\cong  H^i(X,\sO_{X})
  $$
  is surjective for all $i$.
\end{skthm}

In fact, this essentially characterizes Du~Bois singularities shown by the next
theorem. For details see \cite{Kovacs11d}.

\begin{skthm}\label{thm:naive-db}\cite{Kovacs11d}
  Let $X$ be a projective variety over $\bC$.  Then $X$ has only Du~Bois
  singularities if and only if for any $L\subseteq X$ general (global) complete
  intersection subvariety
  $$
  \dim_{\bC} H^i\bigl(L, \sO_L\bigr) \leq \dim_{\bC} Gr^0_FH^{i}\bigl(L, \bC\bigr),
  $$
  where $Gr^0_FH^{i}\bigl(L, \bC\bigr)$ is the graded quotient associated to
  Deligne's Hodge filtration on $H^{i}\bigl(L, \bC\bigr)$.
\end{skthm}

Using \cite[Lemme 1]{MR0376678}, \eqref{thm:hodge-for-db} implies the following:

\begin{skcor}
  \label{cor:coh-inv}
  Let ${{f}}:X\to B$ be a proper, flat morphism of complex varieties with $B$
  connected.  Assume that $X_b$ has only DB singularities for all $b\in B$. Then
  $h^i(X_b, \sO_{X_b})$ is independent of $b\in B$ for all $i$.
\end{skcor}

This will be important later.

The starting point of the precise definition is Du~Bois's construction, following
Deligne's ideas, of the generalized de~Rham complex, which is called the
\emph{Deligne-Du~Bois complex}.  Recall, that if $X$ is a smooth complex algebraic
variety of dimension $n$, then the sheaves of differential $p$-forms with the usual
exterior differentiation give a resolution of the constant sheaf $\bC_X$. I.e., one
has a complex of sheaves,
$$
\xymatrix{%
\sO_X \ar[r]^{d} & \Omega_X^1 \ar[r]^{d} & \Omega_X^2 \ar[r]^{d} & \Omega_X^3
\ar[r]^{d} & \dots  \ar[r]^{d} & \Omega_X^n\simeq \omega_X,
}
$$
which is quasi-isomorphic to the constant sheaf $\bC_X$ via the natural map $\bC_X\to
\sO_X$ given by considering constants as holomorphic functions on $X$. Recall that
this complex \emph{is not} a complex of quasi-coherent sheaves. The sheaves in the
complex are quasi-coherent, but the maps between them are not $\sO_X$-module
morphisms. Notice however that this is actually not a shortcoming; as $\bC_X$ is not
a quasi-coherent sheaf, one cannot expect a resolution of it in the category of
quasi-coherent sheaves.

The Deligne-Du~Bois complex is a generalization of the de~Rham complex to singular
varieties.  It is a complex of sheaves on $X$ that is quasi-isomorphic to the
constant sheaf, $\bC_X$. The terms of this complex are harder to describe but its
properties, especially cohomological properties are very similar to the de~Rham
complex of smooth varieties. In fact, for a smooth variety the Deligne-Du~Bois
complex is quasi-isomorphic to the de~Rham complex, so it is indeed a direct
generalization.

The construction of this complex, $\FullDuBois{X}$, is based on simplicial
resolutions. The reader interested in the details is referred to the original article
\cite{DuBois81}.  Note also that a simplified construction was later obtained in
\cite{Carlson85} and \cite{GNPP88} via the general theory of polyhedral and cubic
resolutions.  An easily accessible introduction can be found in \cite{Steenbrink85}.
Other useful references are the recent book \cite{PetersSteenbrinkBook} and the
survey \cite{Kovacs-Schwede11}. I will actually not use these resolutions here. They
are needed for the construction, but if one is willing to believe the listed
properties (which follow in a rather straightforward way from the construction) then
one should be able to follow the material presented here.

Recently Schwede found a simpler alternative construction of (part of) the
Deligne-Du~Bois complex that does not need a simplicial resolution
\eqref{EasyDuBoisCriterion}.  This allows one to define Du~Bois singularities
\eqref{def:db-sing} without needing simplicial resolutions and it is quite useful in
applications.  For applications of the Deligne-Du~Bois complex and Du~Bois
singularities other than the ones listed here see \cite{SteenbrinkMixed},
\cite[Chapter 12]{Kollar95s}, \cite{Kovacs99,Kovacs00c}.

The word ``hyperresolution'' will refer to either a simplicial, polyhedral, or cubic
resolution. Formally, the construction of $\FullDuBois{X}$ is essentially the same
regardless the type of resolution used and no specific aspects of either types will
be used.

The following definition is included to make sense of the statements of some of the
forthcoming theorems. It can be safely ignored if the reader is not interested in the
detailed properties of the Deligne-Du~Bois complex and is willing to accept that it
is a very close analog of the de~Rham complex of smooth varieties.

\begin{skthm}[{\cite[6.3, 6.5]{DuBois81}}]\label{defDB}
  Let $X$ be a complex scheme of finite type.
  Then there exists a unique object $\Om_X^\mydot \in \Ob D_{\rm filt}(X)$ such that
  using the notation
  $$
  \Om_X^p\leteq Gr^{p}_{\rm filt}\, \Om_X^\mydot[p],
  $$
  it satisfies the following properties
  \begin{enumerate}
  \item 
    \begin{equation*}
      \Om_X^\mydot \qis \bC_{X}.
    \end{equation*}
    \label{resolution}
  \item $\Om$ is functorial, i.e., if $\phi \col Y\to X$ is a morphism of complex
    schemes of finite type, then there exists a natural map $\phi^{*}$ of filtered
    complexes
    $$
    \phi^{*}\col \Om_X^\mydot \to \myR\phi_{*}\Om_Y^{\mydot}.
    $$
    Furthermore, $\Om_X^\mydot \in \Ob \left(D^{b}_{\rm filt, coh}(X)\right)$ and if
    $\phi$ is proper, then $\phi^{*}$ is a morphism in $D^{b}_{\rm filt, coh}(X)$.
    \label{functorial}
  
  \item Let $U \subseteq X$ be an open subscheme of $X$. Then
    $$
    \Om_X^\mydot\resto U 
    \qis\underline{\Omega}^{\,\mydot}_U.
    $$
    
  \item If $X$ is proper, then there exists a spectral sequence degenerating at $E_1$
    and abutting to the singular cohomology of $X$:
    $$
    E_1^{pq}={\bH}^q \left(X, \Om_X^p\right) \Rightarrow H^{p+q}(X, \bC).
    $$\label{item:Hodge}
  
  \item If\/ $\varepsilon_\mydot\col X_\mydot\to X$ is a hyperresolution, then
    $$
    \Om_X^\mydot\qis R{\varepsilon_\mydot}_* \Omega^\mydot_{X_\mydot}.
    $$
    In particular, $h^i\left(\Om_X^p\right)=0$ for $i<0$.
    
  \item There exists a natural map, $\sO_{X}\to \Om_X^0$, compatible with
    {\rm (\ref{defDB}.\ref{functorial})}.
    \label{item:dR-to-DB}

  \item If\/ $X$ is smooth, then
    $$
    \Om_X^\mydot\qis\Omega^\mydot_X.
    $$
    In particular,
    $$
    \Om_X^p\qis\Omega^p_X.
    $$
    
  \item If\/ $\phi\col Y\to X$ is a resolution of singularities, then
    $$
    \Om_X^{\dim X} \qis \myR\phi_*\omega_Y.
    $$
  \end{enumerate}
\end{skthm}

It turns out that the Deligne-Du~Bois complex behaves very much like the de~Rham
complex for smooth varieties. Observe that (\ref{defDB}.\ref{item:Hodge}) says that
the Hodge-to-de~Rham spectral sequence works for singular varieties if one uses the
Deligne-Du~Bois complex in place of the de~Rham complex. This has far reaching
consequences and if the associated graded pieces $\Om_X^p$ turn out to be computable,
then this single property leads to many applications.

The natural map $\sO_{X}\to \Om^0_X$ given by (\ref{defDB}.\ref{item:dR-to-DB}) may
be considered as an invariant of the singularites of $X$. Clearly, if $X$ is smooth,
then it is a quasi-isomorphism, but it may be a quasi-isomorphism even if $X$ is not
smooth. In fact, we are interested in situations when this map is a
quasi-isomorphism.  When $X$ is proper over $\bC$, such a quasi-isomorphism implies
that the natural map
\begin{equation*}
  H^i(X^{\text{an}}, \bC) \rightarrow H^i(X, \sO_{X}) = \bH^i(X, \DuBois{X})
\end{equation*}
is surjective because of the degeneration at $E_1$ of the spectral sequence in
(\ref{defDB}.\ref{item:Hodge}) (cf.\ \eqref{thm:hodge-for-db}).  Notice that this
condition is crucial for proving Kodaira-type vanishing theorems cf.\ \cite[\S
9]{Kollar95s}, \cite[3.H]{Hacon-Kovacs10}.

Following Du~Bois, Steenbrink was the first to study this condition and he christened
this property after Du~Bois. It should be noted that many of the ideas that play
important roles in this theory originated from Deligne. Unfortunately the now
standard terminology does not reflect this.

\begin{skdefini}\label{def:db-sing}
  A scheme $X$ is said to have \emph{Du~Bois} singularities (or \emph{DB}
  singularities for short) if the natural map $\sO_{X}\to \Om^0_X$ from
  (\ref{defDB}.\ref{item:dR-to-DB}) is a quasi-isomorphism.
\end{skdefini}

\begin{skrem}
  If $\varepsilon : X_{\mydot} \rightarrow X$ is a hyperresolution of $X$ then $X$
  has Du~Bois singularities if and only if the natural map $\sO_X \rightarrow \myR
  {\varepsilon_{\mydot}}_* \sO_{X_{\mydot}}$ is a quasi-isomorphism.

  A relative version of this notion for pairs was defined in \cite{Kovacs10a}.
\end{skrem}

\begin{skexample}
  It is easy to see that smooth points are Du~Bois and Deligne proved that normal
  crossing singularities are Du~Bois as well cf.\ \cite[Lemme 2(b)]{MR0376678}.
\end{skexample}

I will finish this section with Schwede's characterization of DB singularities. This
condition makes it possible to define DB singularities without hyperresolutions,
derived categories, etc.  It makes it easier to get acquainted with these
singularities, but it is still useful to know the original definition for many
applications.

\begin{skthm}[{\cite{SchwedeEasyCharacterization}}]
  \label{EasyDuBoisCriterion}
  Let $X$ be a reduced separated scheme of finite type over a field of characteristic
  zero.  Assume that $X \subseteq Z$ where $Z$ is smooth and let $\phi : W
  \rightarrow Z$ be a proper birational map with $W$ smooth and where $Y
  =\phi^{-1}(X)_{\rm {red}}$, the reduced pre-image of $X$, is a simple normal
  crossings divisor (or in fact any scheme with DB singularities).  Then $X$ has DB
  singularities if and only if the natural map $\sO_X \rightarrow \myR \phi_*
  \sO_{Y}$ is a quasi-isomorphism.

  In fact, one can say more.  There is a quasi-isomorphism $\xymatrix{\myR \phi_*
    \sO_{Y} \ar[r]^-{\qis} & \DuBois{X} }$ such that the natural map $\sO_X
  \rightarrow \DuBois{X}$ can be identified with the natural map $\sO_X \rightarrow
  \myR\phi_* \sO_{Y}$.
\end{skthm}

Notice that this condition is the one obtained at the end of the previous section.
Given that and our earlier findings on (semi-) log canonical singularities it may not
come as a surprise that Koll\'ar had conjectured a strong connection between these
singularities. As canonical singularities are rational one should expect a similar
implication between log canonical and Du~Bois:

\begin{skconj}
  [\protect{\cite[1.13]{K+92}} (Koll\'ar's Conjecture)]
  Log canonical singularities are Du~Bois.
\end{skconj}

This conjecture has been recently confirmed in \cite{KK10}. For more see
\S\ref{sec:splitting-principle} and in particular \eqref{thm:main}.

\section{The splitting principle}\label{sec:splitting-principle}

\noindent
The moral of this section can be summarized by the following principle:

\begin{demo*}{\bf {The Splitting Principle}}
  \it Morphisms do not split accidentally.
\end{demo*}

\begin{skrem}
  It is customary to casually use the word ``splitting'' to explain the statements of
  the theorems that follow.  However, the reader should be warned that one has to be
  careful with the meaning of this, because these ``splittings'' take place in the
  derived category, which is not abelian. For this reason, in the statements of the
  theorems below I use the terminology that a morphism admits a \emph{left inverse}.
  In an abelian category this condition is equivalent to ``splitting'' and being a
  direct component (of a direct sum). With a slight abuse of language I labeled these
  as ``Splitting theorems'' cf. \eqref{thm:rtl-crit}, \eqref{thm:db-crit} and
  \eqref{thm:db-criterion}.
\end{skrem}

The first theorem I will recall is a criterion for a singularity to be rational.

\begin{skthm}[\cite{Kovacs00b} ({\sc Splitting theorem I})]\label{thm:rtl-crit}
  Let $\phi:Y\to X$ be a proper morphism of varieties over $\bC$ and
  $\varrho:\sO_X\to\rpforward\phi\sO_Y$ the associated natural morphism.  Assume
  that $Y$ has rational singularities and $\varrho$ has a left inverse, i.e., there
  exists a morphism (in the derived category of $\sO_X$-modules)
  $\varrho':\rpforward\phi\sO_Y\to\sO_X$ such that $\varrho'\circ\varrho$ is a
  quasi-isomorphism of $\sO_X$ with itself.
  Then $X$ has only rational singularities.
\end{skthm}

\begin{skrem}
  Note that $\phi$ in the theorem does not have to be birational or even generically
  finite.  It follows from the conditions that it is surjective.
\end{skrem}

\begin{skcor}\label{cor:split-rtl}
  Let $X$ be a complex variety and $\phi:Y\to X$ a resolution of singularities. If
  $\sO_X\to\rpforward\phi\sO_Y$ has a left inverse, then $X$ has rational
  singularities.
\end{skcor}

\begin{skcor}\label{cor:finite-rtl}
  Let $X$ be a complex variety and $\phi:Y\to X$ a finite morphism. If $Y$ has
  rational singularities, then so does $X$.
\end{skcor}

Using this criterion it is quite easy to prove that log terminal singularities are
rational \eqref{c-elkik}.  For related statements see \cite[5.22]{KM98} and the
references therein.

\begin{skthm}[(= Theorem~\ref{c-elkik})]
  \label{thm:elkik-with-proof}
  Let $X$ be a variety with log terminal singularities. Then $X$ has rational
  singularities.
\end{skthm}

\begin{proof}
{\cite{Kovacs00b}}
The question is local, so one may restrict to a neighbourhood of a point. Then the
index $1$ cover $\pi:\wt X\to X$ is a finite morphism onto $X$. In particular,
$\pi_*$ is exact and the natural morphism $\sO_X\to\pi_*\sO_{\wt X}$ has a left
inverse by the construction of the index $1$ cover.

  Therefore, by \eqref{cor:finite-rtl} it is enough to prove that $\wt X$ has
  rational singularities and so one may assume that $X$ has canonical singularities
  and $\omega_X$ is a line bundle.

  Let $\phi:Y\to X$ be a resolution of singularities of $X$. Since $X$ has canonical
  singularities and $\omega_X$ is a line bundle, there exists a non-trivial morphism
  $$\iota:\myL\phi^*\omega_X\qis \phi^*\omega_X\to \omega_Y.$$ Its adjoint morphism
  on $X$, $\omega_X\to \rpforward\phi\omega_Y$, is a quasi-isomorphism by
  \eqref{lem:can-sings} and \eqref{GR}

  Applying $\myR{\sHom}_Y( \blank , \omega_Y)$ 
  to $\iota$ and using \eqref{thm:adjoint}, one obtains the following diagram which
  defines $\varrho'$:
  $$ 
  \xymatrix{%
    \myR\phi_* \myR{\sHom}_Y(\omega_Y\empty, \omega_Y\empty) \ar[r]^-{\qis} &
    \myR\phi_* \myR{\sHom}_Y(\myL\phi^*\omega_X\empty, \omega_Y\empty) \ar[r]^-{\qis}
    & \myR{\sHom}_X(\omega_X\empty, \myR\phi_*\omega_Y\empty) \ar[d]^-{\qis}
    \\
    \myR\phi_*\sO_Y  \ar[u]_-{\qis} \ar[rr]^{\varrho'} & &\sO_X.  }
  $$ 
  The last quasi-isomorphism uses the fact that $\myR\phi_*\omega_Y\qis \omega_X$.
  It is easy to see that $\varrho'\circ\varrho$ acts trivially on $\sO_X$ and hence
  the statement follows by \ref{thm:rtl-crit} (or \eqref{cor:split-rtl}).
\end{proof}

\noindent
There is a criterion for DB singularities that is similar to the one in
\eqref{thm:rtl-crit}:

\begin{skthm}[\protect{\cite[2.3]{Kovacs99} ({\sc Splitting theorem
      II})}]\label{thm:db-crit} 
  Let $X$ be a complex variety. If $\sO_X\to \DuBois X$ has a left inverse, then $X$
  has DB singularities.
\end{skthm}

\noindent
This criterion has several important consequences. Here is one of them:

\begin{skcor}[\protect{\cite[2.6]{Kovacs99}}]\label{cor:rtl-DB}
  Let $X$ be a complex variety with rational singularities. Then $X$ has DB
  singularities.
\end{skcor}

\begin{proof}
  Let $\phi:Y\to X$ be a resolution of singularities. Then since $Y$ is smooth the
  natural map $\varrho: \sO_X\to \myR\phi_*\sO_Y$ factors through $\DuBois X$ by
  (\ref{defDB}.\ref{item:dR-to-DB}). Then, since $X$ has rational singularities,
  $\varrho$ is a quasi-isomorphism, so one obtains that the natural map $\sO_X\to
  \DuBois X$ has a left inverse. Therefore, $X$ has DB singularities by
  \eqref{thm:db-crit}.
\end{proof}

Recently a few more criterions have been found for DB singularities. The next one
resembles Kempf's criterion for rational singularities \eqref{thm:kempf} and shows that
indeed DB singularities may be considered a close generalization of rational
singularities.

\begin{skthm}[{\cite[3.1]{KSS10}}]
  \label{ThmCMDuBoisCriterion}
  Let $X$ be a normal Cohen-Macau\-lay scheme of finite type over $\bC$.  Let $\phi :
  Y \rightarrow X$ be a resolution of singularities such that the (reduced)
  exceptional set $G$ is a simple normal crossing divisor.  Then $X$ has DB
  singularities if and only if $\phi_* \omega_{Y}(G) \simeq \omega_X$.
\end{skthm}

\noindent
Related results have been obtained in the non-normal Cohen-Macau\-lay case, see
\cite{KSS10} for details.

\begin{skrem}
  The submodule $\phi_* \omega_{Y}(G) \subseteq \omega_X$ is independent of the
  choice of the log resolution.  Thus this submodule may be viewed as an invariant
  that partially measures how far a scheme is from being DB (compare with
  \cite{FujinoNonLCSheaves}).
\end{skrem}

As an easy corollary, one obtains another proof that rational singularities are DB
(this time via the Kempf-criterion for rational singularities).

\begin{skcor}[\protect{\cite{Kovacs00b}}]
  Let $X$ be a complex variety with rational singularities. Then $X$ has DB
  singularities.
\end{skcor}

\begin{proof}
  Since $X$ has rational singularities, it is Cohen-Macau\-lay and normal.  
  Then $\phi_* \omega_{Y} = \omega_X$ but one also has $\phi_* \omega_{Y} \subseteq
  \phi_* \omega_{Y}(G) \subseteq \omega_X$, and thus $\phi_* \omega_{Y}(G) =
  \omega_X$ as well.  The statement now follows from
  Theorem~\ref{ThmCMDuBoisCriterion}.
\end{proof}

One also sees immediately that log canonical singularities coincide with DB
singularities in the Gorenstein case.

\begin{skcor}[\protect{\cite[3.6]{Kovacs99}\cite[3.16]{KSS10}}]
  \label{CorGorLCImpliesDuBois}
  Suppose that $X$ is Gorenstein and normal. Then $X$ is DB if and only if $X$ is log
  canonical.
\end{skcor}

\begin{proof}
  $X$ is easily seen to be log canonical if and only if $\phi_* \omega_{Y/X}(G) \simeq
  \sO_X$.  The projection formula then completes the proof.
\end{proof}

In fact, a slightly jazzed up version of this argument can be used to show that every
Cohen-Macau\-lay log canonical pair is DB:

\begin{skcor}[{\cite[3.16]{KSS10}}]
  CM log canonical singularities are DB.
\end{skcor}

We will see below that it is actually not necessary to assume CM in the previous
theorem. However, the characterization of DB singularities in
\eqref{ThmCMDuBoisCriterion} is still useful on its own.

\begin{skthm}[\protect{\cite[1.6]{KK10} ({\sc Splitting theorem III})}]
  \label{thm:db-criterion}
  Let ${\phi}: Y\to X$ be a proper 
  morphism between reduced schemes of finite type over $\bC$.
  Let $W\subseteq X$ be a closed reduced subscheme with ideal sheaf $\sI_{W\subseteq
    X}$ and $F\leteq {\phi}^{-1}(W)\subset Y$ with ideal sheaf $\sI_{F\subseteq Y}$.
  Assume that the natural map $\varrho$
  $$
  \xymatrix{ \sI_{W\subseteq X} \ar[r]_-\varrho & \myR{\phi}_*\sI_{F\subseteq Y}
    \ar@{-->}@/_1.5pc/[l]_{\varrho'} }
  $$
  admits a left inverse $\varrho'$, that is,
  $\varrho'\circ\varrho=\id_{\sI_{W\subseteq X}}$. Then if $Y,F$, and $W$ all have DB
  singularities, then so does $X$.
\end{skthm}

A somewhat more general version of this was proved in \cite{Kovacs11c}.

This criterion forms the cornerstone of the proof of the following theorem:

\begin{skthm}[\protect{\cite[1.5]{KK10}}]\label{thm:main}
  Let $\phi:{{Y}}\to {{X}}$ be a proper surjective morphism with connected fibers
  between normal varieties. Assume that 
  $Y$ 
  has log canonical singularities and $K_{{Y}}\sim_{\bQ,\phi} 0$.  Then $X$ is DB.
\end{skthm}

\begin{skcor}[\protect{\cite[1.4]{KK10}}]
  Log canonical singularities are DB.
\end{skcor}

\noindent
For the proofs and more general statements, please see \cite{KK10}.  Also, note that
this statement holds in a more general situation, namely it is in fact true that
already semi-log canonical singularities are DB. This is proved in \cite{SingBook}.

\begin{sksubrem}
  Notice that in \eqref{thm:db-criterion} it is not required that ${\phi}$ be
  birational.  On the other hand the assumptions of the theorem and
  \cite[Thm~1]{Kovacs00b} imply that if $Y\setminus F$ has rational singularities,
  e.g., if $Y$ is smooth, then $X\setminus W$ has rational singularities as well.
\end{sksubrem}

This theorem is used in \cite{KK10} to derive various consequences, some of which are
formally unrelated to DB singularities. I will mention some of these in the sequel,
but the interested reader should look at the original article to obtain the full
picture.

\section{Stable families}\label{sec:stable-families}

The connection between log canonical and DB singularities has many useful
applications in moduli theory. I list a few below without proof.

\begin{skthm}[\protect{\cite[7.8,7.9,7.13]{KK10}}]\label{thm:CM-and-all-that}
  Let ${{f}}:X\to B$ be a flat projective morphism of complex varieties with $B$
  connected and such that $X_b$ has log canonical singularities for all $b\in B$.
  Then
  \begin{enumerate}
  \item 
    \label{thm:coh-inv}
    $h^i(X_b, \sO_{X_b})$ is independent of $b\in B$ for all $i$.
  \item 
    \label{thm:cm-inv}
    If one fiber of ${{f}}$ is Cohen-Macau\-lay, then all fibers are Cohen-Macau\-lay.
  \item 
    \label{thm:flat}
    The cohomology sheaves $h^{i}(\omega_{{f}}^\mydot)$ are flat over $B$, where
    $\omega_{{f}}^\mydot$ denotes the {relative dualizing complex} of ${{f}}$.
  \end{enumerate}
\end{skthm}

For arbitrary flat, proper morphisms, the set of fibers that are Cohen-Macau\-lay is
open, but not necessarily closed. Thus the key point of
(\ref{thm:CM-and-all-that}.\ref{thm:cm-inv}) is to show that this set is also closed.

The generalization of these results to the semi log canonical case turns out to be
straightforward, but it needs some foundational work to extend some of the results
used here to the semi log canonical case.  This is done in \cite{Moduli_Book}.  The
general case then implies that each connected component of the moduli space of stable
log varieties parameterizes either only Cohen-Macau\-lay or only non-Cohen-Macau\-lay
objects.

Notice that this still does not mean that one should abandon the non-Cohen-Macau\-lay
objects. There exists smooth projective varieties of general type whose log canonical
model is not Cohen-Macau\-lay and one should naturally prefer to have a moduli space
that includes these. Nevertheless, it is very useful to know that if the general
fiber is Cohen-Macau\-lay, then so is the special fiber.

\eqref{thm:CM-and-all-that} is proved using \eqref{thm:main}, \eqref{cor:coh-inv} and
the following theorem. Before I can state that theorem I need a simple definition.
Let ${{f}}:X\to B$ be a flat morphism.  One says that ${{f}}$ is a \emph{DB family}
if $X_b$ is DB for all $b\in B$.

\newcommand{\clx}[1]{{#1}^{\bdot}}
\newcommand\omegai{h^{-i}(\clx{\omega}_{{{f}}})}
\newcommand\omegait{h^{-i}(\clx{\omega}_{{{f}}_T})}

\begin{skthm}[{\cite[7.9]{KK10}}]\label{thm:coh-base-change}
  Let ${{f}}:X\to B$ be a projective DB family and $\sL$ a relatively ample line
  bundle on $X$.  Then
  \begin{enumerate}
  \item the sheaves $\omegai$ are flat over $B$ for all $i$, 
  \item the sheaves ${{f}}_*(\omegai\otimes \sL^{\otimes q})$ are
    locally free and compatible with arbitrary base change for all $i$ and for all
    $q\gg 0$, and
  \item for any base change morphism $\vartheta: T\to B$ and for all $i$,
    $$
    \big(\omegai\big)_T\simeq \omegait.
    $$
  \end{enumerate}
\end{skthm}

Let me emphasize a special case of this theorem. This had been known for families of
CM varieties, so for instance for stable families of relative dimension at most $2$.

\begin{skcor}\label{cor:base-change-for-omega}
  Let ${{f}}:X\to B$ be a weakly stable family. Then $\omega_{X/B}$ commutes with
  arbitrary base change.
\end{skcor}

For related results that show that this statement is sharp in a certain sense see
\cite{patakfalvi-base-change}.

\section{Deformations of DB singularities}\label{sec:deform-db-sing}

Given the importance of DB singularities in moduli theory it is a natural question
whether they are invariant under small deformation.

It is relatively easy to see from the construction of the Deligne-Du~Bois complex
that a general hyperplane section (or more generally, the general member of a base
point free linear system) on a variety with DB singularities again has DB
singularities. Therefore the question of deformation follows from the following.

\begin{skconj}[{\cite{SteenbrinkMixed}}] %
  \label{conj:DB-defo} %
  Let $D\subset X$ be a reduced Cartier divisor and assume that $D$ has only DB
  singularities in a neighborhood of a point $x\in D$. Then $X$ has only DB
  singularities in a neighborhood of the point $x$.
\end{skconj}

This conjecture was confirmed for isolated Gorenstein singularities by Ishii
\cite{Ishii86}.  Also note that rational singularities satisfy this property, see
\cite{Elkik78}.
\nocite{Kovacs99,Kovacs00c,Ishii85,Ishii87a,Ishii87b,KSS10,MR2644316}

One also has the following easy corollary of the results presented earlier:

\begin{skthm}\label{thm:Gor-defo}
  Assume that $X$ is Gorenstein and $D$ is normal.
  %
  Then the statement of \eqref{conj:DB-defo} is true.
\end{skthm}

\begin{proof}
  The question is local so one may restrict to a neighborhood of $x$. %
  If $X$ is Gorenstein, then so is $D$ as it is a Cartier divisor. Then $D$ is log
  canonical by \eqref{CorGorLCImpliesDuBois}, and then $X$ is also log canonical by
  inversion of adjunction \cite{MR2264806}. (Recall that if $D$ is normal, then so is
  $X$ along $D$).  Therefore $X$ is also DB.
\end{proof}

\begin{skrem}
  It is claimed in \cite[3.2]{Kovacs00b} that the conjecture holds in full
  generality. Unfortunately, the proof published there is not complete. It works as
  long as one assumes that the non-DB locus of $X$ is contained in $D$. For instance,
  one may assume that this is the case if the non-DB locus is isolated.

  The problem with the proof is the following: it is stated that by taking hyperplane
  sections one may assume that the non-DB locus is isolated. However, this is
  incorrect. One may only assume that the \emph{intersection} of the non-DB locus of
  $X$ with $D$ is isolated. If one takes a further general section then it will miss
  the intersection point and then it is not possible to make any conclusions about
  that case.
\end{skrem}

Until very recently the best known result with regard to this conjecture had been the
following:

\begin{skthm}[{\cite[3.2]{Kovacs00b}}]
  Let $D\subset X$ be a reduced Cartier divisor and assume that $D$ has DB
  singularities in a neighborhood of a point $x\in D$ and that $X\setminus D$ has DB
  singularities.  Then $X$ has only DB singularities in a neighborhood of $x$.
\end{skthm}

After submitting this article, but fortunately before it went to press the above
conjecture has been settled by the author of this paper and Karl Schwede:

\begin{skthm}
  [{\cite[4.1,4.2]{Kovacs-Schwede11b}}] %
  Conjecture~\ref{conj:DB-defo} holds, that is: If $D\subset X$ is a reduced Cartier
  divisor such that $D$ has only DB singularities in a neighborhood of a point $x\in
  D$, then $X$ has only DB singularities in a neighborhood of the point $x$.
\end{skthm}

Experience shows that divisors not in general position tend to have worse
singularities than the ambient space in which they reside.  Therefore one would in
fact expect that if $X\setminus D$ and $D$ are nice (e.g., they have DB
singularities), then perhaps $X$ is even better behaved.

We have also seen that rational singularities are DB and at least Cohen-Macau\-lay DB
singularities are not so far from being rational cf.\ \eqref{ThmCMDuBoisCriterion}.
The following result of Schwede supports this philosophical point.

\begin{skthm}[{\cite[5.1]{SchwedeEasyCharacterization}}]
  Let $X$ be a reduced scheme of finite type over a field of characteristic zero, $D$
  a Cartier divisor that has DB singularities and assume that $X\setminus D$ is
  smooth.  Then $X$ has rational singularities (in particular, it is Cohen-Macau\-lay).
\end{skthm}

Let me conclude with a conjectural generalization of this statement:

\begin{skconj}\label{conj:DB-to-rtl}
  Let $X$ be a reduced scheme of finite type over a field of characteristic zero, $D$
  a Cartier divisor that has DB singularities and assume that $X\setminus D$ has
  rational singularities.  Then $X$ has rational singularities (in particular, it is
  Cohen-Macau\-lay).
\end{skconj}

Essentially the same proof as in \eqref{thm:Gor-defo} shows that this is also true
under the same additional hypotheses.

\begin{skthm}
  Assume that $X$ is Gorenstein and $D$ is normal.  Then the statement of
  \eqref{conj:DB-to-rtl} is true.
\end{skthm}

\begin{proof}
  If $X$ is Gorenstein, then so is $D$ as it is a Cartier divisor. Then by
  \eqref{CorGorLCImpliesDuBois} $D$ is log canonical. Then $X$ is also log canonical
  near $D$ by inversion of adjunction \cite{MR2264806}.

  As $X$ is Gorenstein and $X\setminus D$ has rational singularities, it follows that
  $X \setminus D$ has canonical singularities.  Then $X$ has only canonical
  singularities everywhere. This can be seen by observing that $D$ is a Cartier
  divisor and examining the discrepancies that lie over $D$ for $(X, D)$ as well as
  for $X$. Therefore, by \eqref{c-elkik} \cite{Elkik81} $X$ has only rational
  singularities along $D$.
\end{proof}


\appendix

\section{The $\bQ$-Cartier condition in families}\label{app:A}

\newcommand\fe{{\sf f_1}}
\newcommand\fk{{\sf f_2}}


Let $R\subseteq \bP^4$ be a quartic rational normal curve, i.e., the image of the
embedding of $\bP^1$ into $\bP^4$ by the global sections of $\sO_{\bP^1}(4)$.

  Let $T\subseteq \bP^5$ be a quartic rational scroll, i.e., the image of the
  embedding of $\bP^1\times \bP^1$ into $\bP^5$ by the global sections of
  $\sO_{\bP^1\times\bP^1}(1,2)$. Then {$R$ is a hyperplane section of $T$}.
  Indeed, let $\fe$ and $\fk$ denote the divisor classes of the two rulings on $T$
  and let $H\subseteq \bP^5$ be a general hyperplane. Then $C\leteq H\intersect T$ is
  a smooth curve such that $C\sim_T \fe + 2\fk$.  Then by the adjunction formula
  $2g(C)-2=(-2\fe-2\fk+C)\cdot C= -2$, hence $C\simeq \bP^1$. Furthermore, then
  $C^2=4$, so $\sO_{T}(1,2)\resto C\simeq \sO_{C}(4)$. Therefore $C$ is a quartic
  rational curve in $H\simeq \bP^4$, and thus it may be identified with $R$.

  Let $C_R\subseteq \bP^5$ be the projectivized cone over $R$ in $\bP^5$ and
  $C_T\subseteq \bP^6$ the projectivized cone over $T$ in $\bP^6$. Then as $R$ is a
  hyperplane section of $T$, it follows that both $T$ and $C_R$ are hyperplane
  sections of $C_T$, so {$T$ is a smoothing of $C_R$}.

  Let $V\subseteq \bP^5$ be a Veronese surface, i.e., the image of the Veronese
  embedding; the embedding of $\bP^2$ into $\bP^5$ by the global sections of
  $\sO_{\bP^2}(2)$.  Let $D\subset V$ be the image of a smooth conic of $\bP^2$. Then
  $D$ is a hyperplane section of $V$ and it is also a rational normal quartic curve
  in $\bP^4$ so it can also be identified with $R$. Therefore, the same way as above,
  using $C_V$, the cone over $V$, one sees that $V$ is also a smoothing of $C_R$.
  
  It is relatively easy, and thus left to the reader, to compute that $C_R$ has log
  terminal singularities. In particular, this type of singularity is among those that
  appear on stable varieties.  In fact, considering a cyclic covering
  \cite[2.50]{KM98} branched over a highly divisible relatively very ample divisor
  gives a family of stable varieties with the same kind of singularities as the ones
  that appear here. This can be applied for both of the families coming from $C_T$
  and $C_V$.

  The problem this example points to is that if one allows arbitrary families, then
  one may get unwanted results.  For example, using the families derived from $C_T$
  and $C_V$ would mean that $T\simeq \bP^1\times \bP^1$ and $V\simeq \bP^2$ should be
  considered to have the same deformation type (or the same statement for the
  surfaces of general type on the cyclic cover mapping to these fibers). However,
  there are obviously no smooth families that they both belong to, they are
  topologically very different.  For instance, $K_T^2=8$ while $K_V^2=9$.

  The crux of the matter is that $K_{C_T}$ is \emph{not} $\bQ$-Cartier and
  consequently the family obtained from it is not a \emph{(weakly) stable family} as
  defined in \eqref{def:weakly-stable-family} and \eqref{def:stable-family}.  This is
  actually an important point: the canonical classes of the members of the family are
  $\bQ$-Cartier, but the relative canonical class of the family is not $\bQ$-Cartier.
  In particular, the canonical divisors of the members of the family are not
  consistent.

  The family obtained from $C_V$ \emph{has} a $\bQ$-Cartier canonical class and
  consequently ensures that the canonical divisors of the members of the family are
  similar to some extent. Among other things this implies that $K_{C_R}^2=9$.  One
  may also use an actual parametrization of $C_R$ to verify this fact independently.
  It is interesting to note that $K_{C_R}$ is $\bQ$-Cartier, but not Cartier even
  though its self-intersection number is an integer.

\section{The nine lemma in triangulated categories}\label{app:B}

For lack of an appropriate reference the following pseudo-trivial theorem is proved
here for the reader's convenience.

\begin{skthm}\label{thm:nine-lemma}
  Let $\sfA,\sfB,\sfC,\sfA',\sfB',\sfC',\sfA'',\sfB''$ be objects in a triangulated
  category $\frT$ and assume that there exists a commutative diagram in which the
  first two rows and the first two columns form distinguished triangles: \vskip 6em
  \begin{enumerate}
  \item   
    \label{eq:B0}
    \ \vskip -6.5em
    \hfill $
    \xymatrix{ %
      \sfA \ar[r]^{\phi} \ar[d]_\alpha & \sfB \ar[r]^\psi \ar[d]_\beta & \sfC
      \ar[r]^{+1} &
      \\ 
      \sfA' \ar[r]^{\phi'} \ar[d]_{\alpha'} & \sfB' \ar[r]^{\psi'}
      \ar[d]_{\beta'} &    \sfC' \ar[r]^{+1} & 
      \\
      \ar[d]_{+1}   \sfA'' & \ar[d]_{+1} \sfB'' &
      \\
      &&&
    } $ \hfill\ 

    \ 

    \noindent
  Then there exist a morphism $\gamma:\sfC\to\sfC'$ with mapping cone $\sfC''$, i.e.,
  such that $ \xymatrix{\sfC \ar[r] & \sfC'\ar[r] & \sfC'' \ar[r]^{+1} & } $ is a
  distinguished triangle, and morphisms $\sfA''\to\sfB''$, $\sfB''\to \sfC''$,
  $\sfC''\to \sfA''[1]$ such that
  $$
  \xymatrix{%
    \sfA'' \ar[r] & \sfB'' \ar[r] & \sfC'' \ar[r]^{+1} & \\
  }
  $$
  is a distinguished triangle, and the diagram \vskip 6em
\item 
  \label{eq:B1} 
  \ \vskip -6.5em \hfill $ \xymatrix{ %
    \sfA \ar[r]^{\phi} \ar[d]_\alpha & \sfB \ar[r]^\psi \ar[d]_\beta & \sfC
    \ar[r]^{+1} \ar[d]_\gamma &
    \\
    \sfA' \ar[r]^{\phi'} \ar[d]_{\alpha'} & \sfB' \ar[r]^{\psi'} \ar[d]_{\beta'} &
    \sfC' \ar[r]^{+1} \ar[d]_{\gamma'}&
    \\
    \ar[d]_{+1} \ar[r] \sfA'' & \ar[d]_{+1} \sfB'' \ar[r] & \sfC'' \ar[d]_{+1}
    \ar[r]^{+1} &
    \\
    &&& } $\hfill\
  
  \
  
  \noindent
  is commutative.  Furthermore, if the triangulated category $\frT$ is a derived
  category and $\sfC$ and $\sfC'$ are such that $h^i(\sfC)=0$ for $i\neq 0$ and
  $h^j(\sfC')=0$ for $j<0$, then $\gamma$ is uniquely determined by the original
  diagram \rm{(\ref{thm:nine-lemma}.\ref{eq:B0})}.
\end{enumerate}
\end{skthm}

\begin{proof}
  The proof consists of repeated applications of the octachedral axiom. 

  First consider the composition $\sfA\to\sfB\to\sfB'$ and let $D$ be an object that
  completes this morphism to a distinguished triangle.
  $$
  
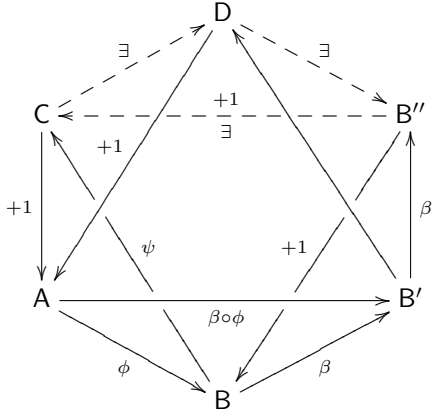
\xymatrix{ %
    && \sfD \ar[dddll]_{+1} \ar@{-->}[drr]^{\exists} & \\
    \sfC \ar[dd]_{+1} \ar@{-->}[urr]^{\exists} &&&& \sfB''
    \ar@{-->}[llll]_{+1}^{\exists}|!{[dd];[llu]}\hole |!{[llu];[lllldd]}\hole
    \ar[dddll]_{+1}|!{[dd];[llu]}\hole|!{[ddllll];[dd]}\hole
    \\
    \\
    \sfA \ar[rrrr]_{\beta\circ\phi} \ar[rrd]_{\phi} &&&& \sfB'
    \ar[lluuu]_(.45){} \ar[uu]_{\beta'}\\
    && \sfB \ar[lluuu]_{\psi} |!{[llu];[rru]}\hole|!{[uuuu];[ull]}\hole
    \ar[rru]_\beta \\
  }
  $$
  Then by the octahedral axiom there exist morphisms as indicated on the above
  diagram such that $\xymatrix{\sfC\ar[r] & \sfD\ar[r] & \sfB\ar[r]^{+1} &}$ is a
  distinguished triangle.

  Next, consider the composition $\sfA\to\sfA'\to\sfB'$. Since
  $\phi'\circ\alpha=\beta\circ\phi$, $D$ is still an object that completes this
  composition to a distinguished triangle.
  $$
  
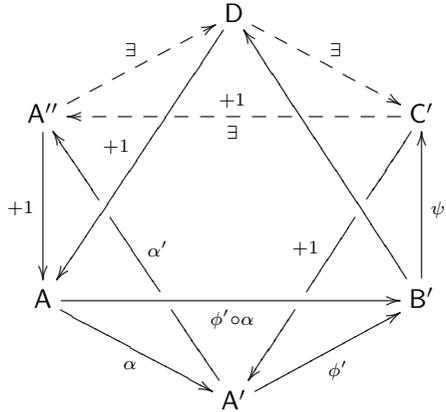
\xymatrix{ %
    && \sfD \ar[dddll]_{+1} \ar@{-->}[drr]^{\exists} & \\
    \sfA'' \ar[dd]_{+1} \ar@{-->}[urr]^{\exists} &&&& \sfC'
    \ar@{-->}[llll]_{+1}^{\exists}|!{[dd];[llu]}\hole |!{[llu];[lllldd]}\hole
    \ar[dddll]_{+1}|!{[dd];[llu]}\hole|!{[ddllll];[dd]}\hole
    \\
    \\
    \sfA \ar[rrrr]_{\phi'\circ\alpha} \ar[rrd]_{\alpha} &&&& \sfB'
    \ar[lluuu]_(.45){} \ar[uu]_{\psi'}\\
    && \sfA' \ar[lluuu]_{\alpha'} |!{[llu];[rru]}\hole|!{[uuuu];[ull]}\hole
    \ar[rru]_{\phi'} \\
  }
  $$
  Then by the octahedral axiom there exist morphisms as indicated on the above
  diagram such that $\xymatrix{\sfA''\ar[r] & \sfD\ar[r] & \sfC'\ar[r]^{+1} &}$ is a
  distinguished triangle.

  Finally, consider the composition $\sfC\to\sfD\to\sfC'$ using the morphisms
  obtained by the above two applications of the octahedral axiom. Let
  $\gamma:\sfC\to\sfC'$ be defined as this composition and $\sfC''$ its mapping
  cone. 
  $$
  
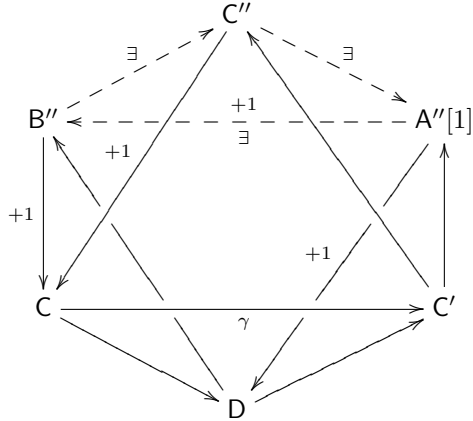
\xymatrix{ %
    && \sfC'' \ar[dddll]_{+1} \ar@{-->}[drr]^{\exists} & \\
    \sfB'' \ar[dd]_{+1} \ar@{-->}[urr]^{\exists} &&&& \sfA''[1]
    \ar@{-->}[llll]_{+1}^{\exists}|!{[dd];[llu]}\hole |!{[llu];[lllldd]}\hole
    \ar[dddll]_{+1}|!{[dd];[llu]}\hole|!{[ddllll];[dd]}\hole
    \\
    \\
    \sfC \ar[rrrr]_{\gamma} \ar[rrd]_{} &&&& \sfC'
    \ar[lluuu]_(.45){} \ar[uu]_{}\\
    && \sfD \ar[lluuu]_{} |!{[llu];[rru]}\hole|!{[uuuu];[ull]}\hole
    \ar[rru]_{} \\
  }
  $$
  Then by the octahedral axiom there exist morphisms as indicated on the above
  diagram such that $\xymatrix{\sfB''\ar[r] & \sfC''\ar[r] & \sfA''[1]\ar[r]^-{+1}
    &}$, and hence $$\xymatrix{\sfA''\ar[r] & \sfB''\ar[r] & \sfC''\ar[r]^{+1} &}$$
  are distinguished triangles. The fact that the diagram
  (\ref{thm:nine-lemma}.\ref{eq:B1}) is commutative follows from the construction and
  the uniqueness of $\gamma$ in the indicated case follows from
  Lemma~\ref{lem:h-nought}.
\end{proof}

\begin{sklem}\cite[2.2.4]{KK10}\label{lem:h-nought}
  Let $\sfC,\sfC'$ objects in a derived category such that $h^{i}(\sfC)=0$ for $i\neq
  0$ and $h^j(\sfC')=0$ for $j<0$. Then any morphism $\gamma:\sfC\to\sfC'$ is
  uniquely determined by $h^0(\gamma)$.
\end{sklem} 
    
\begin{proof}
  By the assumption, the morphism $\gamma:\sfC\to \sfC'$ may be represented by a
  morphism of complexes $\wt\gamma:\wt\sfC\to\what\sfC$, where $\sfC\simeq \wt\sfC$
  such that $\wt\sfC^0= h^0(\sfC)$ and $\wt\sfC^i=0$ for all $i\neq 0$, and
  $\sfC'\simeq \what\sfC$ such that $h^0(\what\sfC)\subseteq \what\sfC^0$. However
  $\wt\gamma$ has only one non-zero term, $h^0(\gamma)$.
\end{proof}


\def\cprime{$'$} \def\polhk#1{\setbox0=\hbox{#1}{\ooalign{\hidewidth
  \lower1.5ex\hbox{`}\hidewidth\crcr\unhbox0}}} \def\cprime{$'$}
  \def\cprime{$'$} \def\cprime{$'$} \def\cprime{$'$}
  \def\polhk#1{\setbox0=\hbox{#1}{\ooalign{\hidewidth
  \lower1.5ex\hbox{`}\hidewidth\crcr\unhbox0}}} \def\cdprime{$''$}
  \def\cprime{$'$} \def\cprime{$'$} \def\cprime{$'$} \def\cprime{$'$}
\providecommand{\bysame}{\leavevmode\hbox to3em{\hrulefill}\thinspace}
\providecommand{\MR}{\relax\ifhmode\unskip\space\fi MR}
\providecommand{\MRhref}[2]{%
  \href{http://www.ams.org/mathscinet-getitem?mr=#1}{#2}
}
\providecommand{\href}[2]{#2}

\end{document}